\documentclass{amsart}

\usepackage{mathabx,hyperref,mathtools,mathrsfs,amsrefs,comment}
\usepackage{amsfonts,amssymb,amsthm}
\usepackage{color,url,hyperref}
\usepackage{tikz}

\usepackage{xypic}
\usepackage{accents}
\usetikzlibrary{cd}
\DeclareSymbolFont{bbold}{U}{bbold}{m}{n}
\DeclareSymbolFontAlphabet{\mathbbold}{bbold}

\newcommand\IH{{\mathbb{H}}} 
\newcommand\C{{\mathbb{C}}}

\newcommand\Q{{\mathbb{Q}}}
\newcommand\R{{\mathbb{R}}}
\newcommand\N{{\mathbb{N}}}
\newcommand\Z{{\mathbb{Z}}}
\newcommand\Pone{\mathbb{P}^1}
\newcommand{\rs}{\widehat{\C}}
\newcommand{\Aut}{\text{Aut}}

\newcommand{\Teich}{\text{Teich}}
\newcommand{\PMod}{\text{PMod}}
\newcommand{\Moduli}{\text{Moduli}}

\newcommand{\Homeo}{\text{Homeo}}

\newtheorem{thm}{Theorem}[section]
\newtheorem*{nonumthm}{Theorem}
\newtheorem{conj}[thm]{Conjecture}
\newtheorem{cor}[thm]{Corollary}
\newtheorem{lem}[thm]{Lemma}
\newtheorem{prop}[thm]{Proposition}
\newtheorem{ex}[thm]{Example}

\theoremstyle{definition}
\newtheorem{defn}{Definition}[section]

\theoremstyle{remark}

\newcommand{\sft}{{\sf t}}
\newcommand{\sfT}{{\sf T}}
\newcommand{\sfX}{{\sf X}}
\newcommand{\sfx}{{\sf x}}
\newcommand{\sfy}{{\sf y}}

\newcommand{\sfz}{{\sf z}}
\newcommand{\sfv}{{\sf v}}

\newcommand{\cG}{\mathcal{G}}
\newcommand{\cT}{\mathcal{T}}
\newcommand{\cS}{\mathcal{S}}
\newcommand{\calL}{\mathcal{L}}
\newcommand{\cJ}{\mathcal{J}}
\newcommand{\cN}{\mathcal{N}}
\newcommand{\cM}{\mathcal{M}}

\newcommand{\cW}{\mathcal{W}}

\newcommand{\cV}{\mathcal{V}}
\newcommand{\cO}{\mathcal{O}}

\newcommand{\Tw}{\text{Tw}}
\newcommand{\twist}{\text{tw}}

\newcommand{\curves}{\text{\sc{Curves}}}

\newcommand{\mc}{\text{\sc{MultiCurves}}}

\newcommand{\End}{\text{End}}
\newcommand{\whatsigma}{\widehat{\sigma}}
\newcommand{\wttau}{\widetilde{\tau}}

\newcommand{\wtmu}{\widetilde{\mu}}

\newcommand{\op}{^{\text{op}\!}}
\newcommand{\za}{\alpha}

\newcommand{\zf}{\phi}

\newcommand{\zg}{\gamma}
\newcommand{\zG}{\Gamma}
\newcommand{\zi}{\iota}

\newcommand{\zp}{\pi}

\newcommand{\zr}{\rho}
\newcommand{\zt}{\tau}

\newcommand{\bbZ}{\mathbb{Z}}

\newcommand{\co}{\colon
}

\newcommand{\diam}{\text{diam}}
\newcommand{\Stab}{\text{Stab}}
\newcommand{\wtgamma}{\widetilde{\gamma}}

\newcommand{\wtalpha}{\widetilde{\alpha}}

\begin{document}

\author{Walter Parry}
\email{walter.parry@emich.edu}
\author{Kevin M. Pilgrim}
\email{pilgrim@iu.edu}

\title{Characterizations of contracting Hurwitz bisets}

\begin{abstract}
A critically finite branched self-cover $f: (S^2, P) \to (S^2, P)$ determines naturally three iterated function systems: one on the pure mapping class group $\cG:=\PMod(S^2, P)$, one on the Teichm\"uller space $\cT:=\Teich(S^2, P)$, and one on a finite-dimensional real vector space $\cV$.  We show that contraction for any one of these systems implies contraction for the others.  
 
 \end{abstract}



\date{\today}

\maketitle
\setcounter{tocdepth}{1}

\tableofcontents

\newpage

\section{Introduction}
\label{secn:intro}

A \emph{Thurston map} $f: (S^2, P) \to (S^2, P)$ is an orientation-preserving, branched covering map of degree $d \geq 2$ whose set of branch values is contained in a finite set $P$ for which $f(P)\subset P$. A celebrated characterization  result of W. Thurston \cite{DH1} identifies those Thurston maps which are conjugate up to isotopy relative to $P$ to a rational function--these we call \emph{unobstructed}.  A corresponding rigidity result says that with a well-understood family of exceptions, such a rational function is unique up to holomorphic conjugacy.  The exceptions are Thurston maps with so-called Euclidean orbifold, and we do not discuss them here.  

An \emph{iterated function system} (IFS) is a set together with a
family of self-maps.  Under composition, it generates a semigroup of
self-maps.  The IFSs that arise in our setting are defined on complete locally compact unbounded metric spaces.  We are concerned with large-scale behavior, and the following notion is central to our development:
\begin{defn}
\label{defn:contracting}
An IFS on a complete, locally compact, unbounded metric space is \emph{contracting} if there is a compact subset, or \emph{attractor}, into which every orbit eventually lands. 
\end{defn}

In this work, we show how a Thurston map $f: (S^2, P) \to (S^2, P)$ leads naturally to three iterated function systems:  one on the pure mapping class group $\cG:=\PMod(S^2, P)$, one on the Teichm\"uller space $\cT:=\Teich(S^2, P)$, and one on a finite-dimensional vector space $\cV$.  Our main result, Theorem \ref{thm:main} below, says that contraction for any one of these IFSs implies contraction for the others.  

The IFSs we associate to $f$ are actually invariants not of the single map $f$, but rather of a class of maps naturally associated to $f$, which we now describe.  Given a Thurston map $f$, its associated \emph{pure augmented Hurwitz class} is
\[ \frak{H}:=\{g_0\circ f \circ g_1 : g_0, g_1 \in \Homeo^+(S^2, P)\}.\]
The set of isotopy classes of elements of $\frak{H}$ relative to $P$ is therefore a countable set, $B$.  Taking $n$-fold compositions, we obtain countable sets $B^{\circ n}$, $n=1, 2, 3, \ldots$, and indeed $B$ generates a semigroup, $B^*$. 
The action of pre- and post-composition by orientation-preserving homeomorphisms  on $\frak{H}$ descends to an action of the pure mapping class group $\cG:=\PMod(S^2, P)$ on $B$. Thus associated to $f$ is a natural $\cG$-\emph{biset}, which we call its associated \emph{Hurwitz biset}, denoted $B$.   The IFSs we introduce below are invariants of its Hurwitz biset; they do not depend on a chosen representative $f\in B$. A Thurston map is non-exceptional if and only if every element of $\frak{H}$ is non-exceptional; in this case, we also say that $B$ is non-exceptional, and we deal exclusively with such $B$.  Our perspective here is a shift in focus, away from the behavior of a single map $f$, and towards the behavior of its Hurwitz biset $B$.  This line of inquiry was first suggested explicitly by Kameyama \cite{kameyama:equivalence}.  We now briefly describe each of these IFSs in turn.

\subsection{IFS on $\cG$.} 
\label{subsecn:G}The Hurwitz $\cG$-biset $B$ has the property that the right action by pre-composition is free, with finitely many orbits. Setting $\sfX$ to be a right-orbit transversal, or \emph{basis} for $B$, it follows that given any $\sfx \in \sfX$ and $g \in \cG$, there are unique $\sfy \in \sfX$ and $h\in \cG$ with the property that $g\circ \sfx = \sfy \circ h$. Denoting $h:=x@g$, we obtain a collection $\{g \mapsto \sfx @ g : \sfx \in \sfX\}$ of self-maps of $\cG$, and thus an IFS on $\cG$. Equipping $\cG$ with any word metric, one says that $B$ is contracting over $\cG$ if this IFS is contracting.  This property is independent of the choice of metric and of basis; see \cite[Prop. 4.3.12]{MR4475129}. 

\subsection{IFS on $\cT$}
\label{subsecn:T} Any Thurston map $f: (S^2, P) \to (S^2, P)$ induces,
via pulling back complex structures, a holomorphic self-map $\sigma_f:
\cT \to \cT$ of the Teichm\"uller space $\cT:=\Teich(S^2, P)$. Applying
this observation to each of the elements comprising a basis $\sfX$
for $B$, we obtain a collection $\{\sigma_\sfx: \sfx \in \sfX\}$ of
self-maps of $\cT$, and thus an IFS on $\cT$. When $B$ is non-exceptional, the map $\sigma_f$ is a contraction\footnote{Technically, one must either pass to the $|P|$th  iterate, or average the metric over such iterates; see \S \ref{subsecn:contraction}.} 
with respect to a suitable compatible metric on $\cT$,
though it is usually not uniformly contracting, since the space
$\cT$ is non-compact. It turns out that contraction of this IFS is equivalent to the limit space of a certain algebraic correspondence on moduli space being compact, and this in turn is independent of the chosen basis, so that contraction of the IFS is independent of the chosen basis.
See \S\S
\ref{subsecn:limitsets} and  \ref{subsecn:codingtree}. 

\subsection{IFS on a finite-dimensional vector space $\cV$}
\label{subsecn:V} A \emph{multicurve} is a finite collection $\Gamma$
of isotopy classes of simple, closed, unoriented, pairwise disjoint,
pairwise non-homotopic, essential, nonperipheral curves in $S^2-P$. We
allow such a collection to be empty.  A Thurston map $f$ induces a
pullback function $\Gamma \mapsto f^{-1}(\Gamma)$ on multicurves, and
a linear map $L[f,\Gamma]: \R[\Gamma] \to \R[f^{-1}(\Gamma)]$. The
group $\cG$ acts on the right, via pullback, on the set  of multicurves;  there are finitely many orbits.  Let $\sfT$ denote an orbit transversal to this action, and set $\cV:=\oplus_{\sft \in \sfT }\R[\Gamma_{\sft}]$. We show that each pair $(\sfx, \sft) \in \sfX \times \sfT$ induces an element of $\End(\cV)$. We obtain a collection $L[B]$ of self-maps of $\cV$, and thus an IFS on $\cV$.  A basic invariant is then its \emph{joint spectral radius}, $\whatsigma(L[B])$.  The contraction property is then equivalent to the condition that $\whatsigma(L[B])<1$; see \S \ref{secn:3to1}. 

The IFS on $\cV$ admits a faithful combinatorial encoding via what we
call here the \emph{strata scrambler}--a directed weighted graph with
nodes given by elements of $\sfT$, and edges weighted by matrices for
the endomorphisms induced by pairs $(\sfx, \sft)$.  The strata
scrambler is algorithmically computable.  Indeed, for maps with four or five postcritical points, and low
enough degree, there are effective implementations by the first author
\cite{netmapsite}, \cite{pf4site}.  See Section \ref{sec:examples} for
some examples.

\subsection*{Main result} We may now state our main result:

\begin{thm}
\label{thm:main}
Suppose $B$ is a non-exceptional Hurwitz biset of Thurston maps.  Then the following are equivalent:
\begin{enumerate}
\item The IFS on $\cG$ given in \S (\ref{subsecn:G}) is contracting.
\item The IFS on $\cT$ given in \S (\ref{subsecn:T}) is contracting.
\item The IFS on $\cV$ given in \S (\ref{subsecn:V}) is contracting.
\end{enumerate}
\end{thm}

\subsection{Contraction and rationality} W. Thurston's characterization of rational functions \cite{DH1} yields the following; see also Proposition \ref{prop:ratB} below, which gives more general statements.

\begin{thm}
\label{thm:joint}
For a non-exceptional Hurwitz biset $B$, any (equivalently, each) of the conditions (1), (2), (3),  in Theorem \ref{thm:main} imply 
\begin{enumerate}
\setcounter{enumi}{3}
\item every element of $B^*$ is unobstructed.  
\end{enumerate}
\end{thm}

Unfortunately the converse to Theorem \ref{thm:joint} eludes us. The obstacle seems to be that the spectral radii of elements of the semigroup generated by $L[B]$ could all be strictly less than one, while the joint spectral radius $\whatsigma(L[B])$ could be equal to one.  This phenomenon can occur for two-by-two nonnegative real matrices \cite{MR2003315}.   Even for matrices with rational entries, such as those arising here, it is known to be algorithmically undecideable if the joint spectral radius is less than or equal to one.  There is substantial work on this and related topics; see \cite{Cicone:2015aa} for a survey.  Nevertheless, we make the following

\begin{conj}
\label{conj:converse}
Suppose a non-exceptional Hurwitz biset has the property that every element of $B^*$ is unobstructed.  Then $B$ is contracting. 
\end{conj}

We can show Conjecture \ref{conj:converse} holds in special cases.

A \emph{topological polynomial} is a Thurston map possessing a fixed
critical point of maximal degree, corresponding to the point at
infinity of an ordinary polynomial.  A topological polynomial $f:
(S^2, P) \to (S^2, P)$ for which every cycle in $P$ contains a
critical point is conjugate-up-to-isotopy to a hyperbolic complex
polynomial, by the Berstein-Levy criterion \cite[Theorem 10.3.9]{MR3675959}. 

 In \S \ref{sec:examples}, we supply a proof of a folklore result on
the structure of completely invariant multicurves for topological
polynomials.  We then show as an application of Theorem\ref{thm:main}:

\begin{thm}\label{thm:toppoly}
Suppose $\#P \geq 4$ and $B$ is the Hurwitz biset of a topological polynomial $f:(S^2,P) \to (S^2, P)$ satisfying either of the following two conditions:
\begin{enumerate}
\item each cycle in $P$ contains a critical point;
\item there is a unique cycle in $P$ not containing a critical point, and the length of this cycle is equal to one.
\end{enumerate} 
Then $\widehat{\sigma}(L[B])<1$, so $B$ is contracting. 
\end{thm}

Theorem \ref{thm:toppoly}, and its consequences when combined with
Theorem \ref{thm:main}, are related to previous work, which we next
describe.

Work of Buff-Epstein-Koch \cite{koch:endos} shows that in the setting of Case 1 of Theorem \ref{thm:toppoly}, for the special case of polynomials with the property that each critical point is periodic, the algebraic correspondence on moduli space extends to an endomorphism of projective space, the algebraic degeneracy locus is the postcritical locus of this endomorphism, its complement is Kobayashi hyperbolic, and the dynamics on the degeneracy locus is repelling.   This yields compactness of the limit space of the algebraic correspondence, and by standard arguments, contraction of its Hurwitz biset $B$.  

In \cite[\S 7]{bartholdi:nekrashevych:twisted}, it is shown that the correspondence on moduli space induced by a quadratic polynomial $f$ for which the unique finite critical point maps to a repelling fixed-point after two iterations is, up to conjugacy, given by a hyperbolic rational function $g$.  It follows that the biset $B$ of $f$ is contracting.  This is a special case of the setting of Case (2) of Theorem \ref{thm:toppoly}. 

Nekrashevych \cite[Thm. 7.2]{MR3199801} notes that the biset $B$ for an arbitrary hyperbolic polynomial $f$ is contracting, by building an affine contracting model.  However, the arguments given there do not yield contraction of the IFS on Teichm\"uller space. 

\subsection{Contraction and the conjugacy problem} Suppose $B$ is a Hurwitz biset of Thurston maps, as above.  Two elements $f_1, f_2\in B$ are called \emph{conjugate} if there exists $g \in \cG$ with $g\circ f_1 = f_2 \circ g$.  This is almost, but not quite, equivalent to conjugacy-up-to-isotopy relative to $P$, since we are restricting to the pure mapping class group $\cG$, and a general conjugacy-up-to-isotopy might nontrivially permute the elements of $P$. If $B$ is contracting, then given any pair of elements of $B$, iteration of the so-called \emph{symbolic pullback algorithm} reduces in logarithmic time the problem of deciding when two elements $f_1, f_2$  are conjugate to the analogous problem where now $f_1, f_2$ lie in a finite set comprising a global attractor for this algorithm \cite[\S 7.2]{MR4213769}. It is therefore of interest to understand when a Hurwitz biset $B$ of Thurston maps is contracting over $\cG$. 

\subsection{Rationality and the diversity of conjugacy classes}   
One  motivation for our study of Hurwitz bisets $B$ is the diverse range of behavior, as $B$ varies, for the set of conjugacy classes in $B$.  The examples at the end of Section~\ref{secn:hb} illustrate various possibilities.   

When $B$ is contracting, every element $f$ of $B$, and indeed every
element of $B^*$, is unobstructed (Proposition \ref{prop:ratB}).  Since $B$ is assumed non-exceptional, for each $n$, there are only finitely many conjugacy classes in each $B^{\circ n}$.  This is reminiscent of the well-known fact that apart from a few low-complexity cases, the mapping class group of a closed surface marked at finitely many points contains only finitely many conjugacy classes which are realizable by holomorphic automorphisms with respect to some complex structure on the surface \cite[\S 7.2]{fm}.

When $B$ is not contracting, the distribution and nature of conjugacy classes in $B$ can be rich and varied. For example, the composition of two unobstructed elements of a Hurwitz biset
$B$ need not be unobstructed (Example \ref{ex:rnc}).  Even if every
element of a Hurwitz biset $B$ is rational, there can be obstructed
elements in $B^*$ (Examples \ref{ex:cubicpoly} and \ref{ex:fixed}). 

\subsection{Notes} 

\begin{enumerate}
\item Ramadas \cite{ramadas2025thurstonobstructionstropicalgeometry} formulates a tropical notion of a closely related correspondence on moduli space to the one considered here.  The tropical spaces involved are polyhedral objects obtained by gluing convex cones together along faces. The tropical moduli space correspondence constructed yields a piecewise linear multivalued action on the tropical moduli space. Proposition 7.8 there asserts that every locally-defined, single-valued brranch has the same matrix as some Thurston linear transformation. It seems that this tropical correspondence contains essentially the same data as the scrambler considered here.

\item The three conditions in Theorem \ref{thm:main} are manifestations of a strong, uniform hyperbolicity property. It is tempting to look for mild relaxations of this condition.

A natural algebraic notion of subhyperbolicity is the property that the faithful quotient of the iterated monodromy group action associated to a biset $B$ over $\cG$ is contracting; see \cite[\S 3.6]{nekrashevych:combinatorics}.   A natural analytic notion of subhyperbolicity is that the IFS on Teichm\"uller space be uniformly contracting with respect to a perhaps incomplete suitable orbifold metric compatible with its topology.  When $\#P=4$, there are close relationships between these properties.  
\end{enumerate}

\subsection{Outline}

\begin{itemize}

\item \S \ref{secn:bisets} reviews background from the general theory of
bisets, and proves some preparatory results needed later.

\item \S \ref{secn:hb} applies the constructions in \S \ref{secn:bisets} to the setting of Hurwitz
bisets.  The IFS on $\cG$ is defined.  Two technical subtleties emerge.  First, it turns out to be more natural to look at the opposite group ${\op \cG}$, so that the free action is a left action.  Second, to handle composition, we technically work
in $B^{\otimes n}$--a related biset in which the compositions
remember, to some extent, the factorization of elements into $n$-fold
compositions of elements of $B$.  
Examples~\ref{ex:euclidean} and~\ref{ex:degree3} give Hurwitz
bisets for which $B^{\otimes 2} \neq B^{\circ 2}$. Examples \ref{ex:finrat}, \ref{ex:finratfinnonrat},
\ref{ex:finratinfnonrat}, and \ref{ex:infnonrat} showcase the diversity of deployment of conjugacy classes. 

\item  \S \ref{secn:scrambler} defines the IFS on $\cV$ and discusses the joint spectral radius.

\item \S \ref{secn:twistgroups} gives the proof that contraction on
$\cG$ (condition (1)) implies contraction on $\cV$ (condition (3)) in
Theorem \ref{thm:main}.

\item \S \ref{sec:corr} introduces the IFS on Teichm\"uller space and the algebraic correspondences on moduli space induced by a Hurwitz biset.  A key perspective is that upon identifying $S^2$ with the complex projective line $\Pone$, a Hurwitz biset whose elements are maps is canonically identified with
one whose elements are certain paths in moduli space.  The latter interpretation is needed to analyze the IFS on Teichm\"uller space via a coding tree. It concludes with the proof, following standard arguments, that contraction on $\cT$ (condition (2)) implies contraction on $\cG$ (condition (1)) in Theorem \ref{thm:main}. 

\item \S \ref{secn:3to1} completes the proof of Theorem \ref{thm:main} by showing contraction on $\cV$ (condition (3)) implies contraction on $\cT$ (condition (2)), using
both results of Selinger \cite{selinger:pullback} and the machinery of \S \ref{sec:corr}.  
\end{itemize}

\subsection{Notation and conventions.}  All maps between surfaces are assumed to be orientation-preserving. Composition of maps is written $f \circ g$ where $g$ is performed first. Concatenation of paths is written $\alpha \cdot \beta$ so that $\alpha$ is traversed first. The action of a group element $g$ on an element $x$ of a set $X$ is written either $g.x$ or $x.g$ depending on whether we consider left or right actions.

\subsection{Acknowledgements} We thank L. Bartholdi for useful conversations.  Kevin M. Pilgrim acknowledges support from NSF DMS grant no 2302907, Universit\"at Saarlandes, and the Max-Planck Institut f\"ur Mathematik Bonn.

\section{Bisets}
\label{secn:bisets}

We briefly collect a few facts and definitions; see \cite{nekrashevych:book:selfsimilar} (where the term ``permutational bimodule'' is used instead of biset), \cite{MR4475129},  and \cite[\S 3]{MR3852445}.

Suppose $G, H$ are groups, and $B$ denotes now an arbitrary set. A $(G,H)-$biset is a set $B$ with a pair of
commuting actions, of $G$ on the left, and $H$ on the right; written
$gbh$ where $g \in G$, $h\in H$ and $b \in B$. Here we are exclusively
concerned with the case $G=H$, and call the set $B$ equipped with
these actions a $G$-biset for short. A $G$-biset is \emph{left-covering} if
the left action is free with finitely many orbits. A biset is
\emph{irreducible} if for any $b \in B$, we have $\{g_0bg_1: g_0, g_1
\in G\}=B$.  We deal exclusively with countable groups and countable
bisets here, unless otherwise mentioned.  

We prefer to deal with left-covering bisets.  Later, this will introduce some lexicographic complications.

A left-covering biset has a finite \emph{basis} $\sfX \subset B$ so
that each $f \in B$ is uniquely expressible as $f=g\sfx$ for some $g
\in G, \sfx \in \sfX$.  Assuming that $B$ is also irreducible, $\sfX$ may
be computed as follows.  Since $B$ is a covering biset, there are
finitely many, say $D$, left orbits.  The group $G$ acts transitively
on the right on the finite set of left orbits. Suppose $f \in B$ is
arbitrary. Set $G_f:=\text{Stab}_G G.f$; thus for each $g \in G_f$ we
have $f g = h f$ for some unique $h \in G$.  Choose coset
representatives $g_1,\dotsc,g_D$ so that $G=G_f g_1 \sqcup \cdots
\sqcup G_f g_D$.  Then $\{\sfx_i:=f g_i: i=1,\dotsc,D\}$ is a basis of
$B$.

A basis determines a finite collection of operators on $G$.  Given
$\sfx \in \sfX$ and $g \in G$, we have $\sfx g = h \sfy$ for unique $h
\in G$ and $\sfy \in \sfX$; we set $\sfx @g := h$.  The collection of
operators \[ g \mapsto \sfx @g, \; \sfx \in \sfX, \; g \in G\] then
defines an iterated function system (IFS) on $G$.  The dynamics of
this IFS will be of fundamental importance.

Bisets may be iterated:
\[ B^{\otimes n}:=\{f_1 \otimes f_{2}\otimes \cdots \otimes f_n : f_i \in B\}; \]
here the elements are equivalence classes in which 
\[ (f_1, f_{2}, \ldots, f_n) \sim (f_1 g_1, g_1^{-1}f_2g_2, \ldots, g_{n-2}^{-1}f_{n-1}g_{n-1}, g_{n-1}^{-1}f_n), \; g_i \in G,\]
and the left- and right-actions of an element $g$ are induced by
multiplication of $f_1$ on the left and of $f_n$ on the right,
respectively. If $\sfX$ is a basis of $B$ then $\sfX^{\otimes
n}=\{\sfx_1\otimes \cdots \otimes \sfx_n : \sfx_i \in \sfX\}$ is a
basis for $B^{\otimes n}$ and this may be identified with the set
$\sfX^n$ of words of length $n$ in the alphabet $\sfX$.  
Setting $\emptyset $ to be the empty word, with the operation of
concatenation of words, the set $\sfX$ generates a semigroup $\sfX^* =
\cup_{n\geq 0}\sfX^n$; for $\sfv \in \sfX^*$, we denote by $|\sfv|$
the length of $\sfv$.  Any covering biset also generates a semigroup
$B^{\otimes *}$ under taking tensor powers.

The semigroup $\sfX^*$ acts on the left on $G$, so that $(\sfx \cdots \sfy)@g=\sfx @\cdots (\sfy@g)$, yielding an IFS on $G$.

We will later need  some related notation.  For $H, K \subset G$ and $n \in \N$, we set 
\[ HK : = \{hk : h \in H, k \in K\}\subset G,\]
\[ \sfX^n H:=\{\sfv h: \sfv \in \sfX^n, h \in H\} \subset
B^{\otimes n},\]
\[ H\sfX^n:=\{h \sfv : h \in H, \sfv \in \sfX^n\} \subset B^{\otimes n},\]
\[ \sfX^n(H):=\{\sfv@h : \sfv \in \sfX^k, h \in H\}\subset G.\]
Note that
\[ \sfX^n(HK) \subset \sfX^n(H)\sfX^n(K).\]

\begin{defn}
\label{defn:contractingbiset}
The $G$-biset $B$ is \emph{contracting (on $G$)} or \emph{hyperbolic} if for some (equivalently, any) basis $\sfX \subset B$, there exists a finite set $\cN \subset G$ such that for every $g \in G$ and every $n \in \N$ sufficiently large, we have $\sfX^ng \subset \cN\sfX^n$.
\end{defn}

In other words, the biset $B$ is contracting on $G$ if and only if the iterated function system on $G$, defined above,
has the property that all orbits eventually land in a finite set $\cN$, i.e. is contracting in the sense of Definition \ref{defn:contracting}. 

Similar facts hold for right-free covering bisets, \emph{mutatis mutandis}. 

\section{Hurwitz bisets and IFS on $\cG$}
\label{secn:hb}

Fix a finite subset $P \subset S^2$ with $\#P \geq 3$.  Let $\cG$ denote
the pure mapping class group $\PMod(S^2, P)$, and let $B$ be a Hurwitz
class of Thurston maps as defined in \S \ref{secn:intro}.  With the operations of pre- and post-composition, $B$ is naturally a $\cG$-biset.   Choosing a basis $\sfX$ for $B$, we obtain an IFS on $\cG$ generated by the operators $\{g \mapsto \sfx@g : \sfx \in \sfX\}$.

The Hurwitz $\cG$-biset $B$ has a technical defect:  the right
action by pre-composition is free, instead of the left action as per
our preference in \S \ref{secn:bisets}.  We resolve this issue by
introducing the notion of the opposite group.  This will later be
natural, since the mapping class group, viewed as the fundamental
group of moduli space, acts by pre-concatenation; see \S \ref{sec:corr}.

\subsection*{Opposite groups.} For a group $H$ we define the \emph{opposite group} ${\op H}$ as follows.
As a set ${\op H}=H$.  The operation $\cdot $ of ${\op H}$ is defined so
that $g\cdot h=h\circ g$, where $g,h\in {\op H}$ and $\circ $ is the
operation of $H$.  One readily verifies that this definition makes
${\op H}$ a group anti-isomorphic to $H$ and $\op({\op H})$ is naturally isomorphic
to $H$.

Given a finite subset $P \subset S^2$, we denote by ${\op \cG}:={\op\PMod(S^2, P)}$ the pure mapping class group, equipped with its opposite group structure. We are interested in ${\op \cG}$ because later (\S \ref{subsec:mcgact}) we model the action of $\cG$ on
Teichm\"{u}ller space via pullback of complex structures, so that we obtain a left action of ${\op \cG}$ on Teichm\"uller space.  

\subsection*{Opposite Hurwitz biset.} Similarly, Thurston maps induce a map on Teichm\"uller space by pulling back complex structures (\S \ref{subsec:liftB}).  To turn this into a left action of a semigroup, we define the \emph{opposite Hurwitz biset} ${\op B}$ as the ${\op \cG}$-biset whose underlying set is $B$ and in which the operations are given by 
\[ g_1 \cdot f \cdot g_0 :=g_0 \circ f \circ g_1, \quad f \in B, \; g_0, g_1 \in {\op \cG}.\]
So ${\op B}$ is a ${\op \cG}$-biset with the conventions in
Section~\ref{secn:bisets}. The $\cG$-biset $B$ is contracting if and only if the ${\op \cG}$-biset ${\op B}$ is contracting.

\subsection*{Iteration.}  Iteration leads to several closely related, but distinct, bisets.  Fix $n \in \N, n \geq  2$.  

The simplest one is the $\cG$-biset 
\[ B^{\circ n} := \{ f_1 \circ \cdots \circ f_n: f_i \in B\}\]
with the action of $\cG$ by pre- and post-composition. As we previously mentioned, however, this is right-free, not left-free.  But, as is the case with $B$ itself, we may view $B^{\circ n}$ as an ${\op \cG}$-biset as well. 

Later, we want elements of $B$ to ``act'' on the left on various sets of objects (multicurves, Teichm\"uller space, etc.) via pullback.  
It is therefore natural to look also at 
\[ {\op B}^{\circ n}: = \{ f_1 \cdots f_n := f_n \circ \cdots \circ f_1 : f_i \in {\op B}\}\]
with the action of ${\op \cG}$ given by 
\[ g_1\cdot (f_1 \cdots f_n)\cdot g_0:=g_0 \circ (f_n \circ \cdots \circ f_1)\circ g_1.\]
so that again the left action is free. 

Given a basis $\sfX$ of $B$, we wish $\sfX^n$ to be a basis of $B^n$, and for this, it is necessary to work with tensor products, and not just compositions, as we will shortly show. So yet another iterate is the $\cG$-biset 
\[ B^{\otimes n}:=\{ f_1 \otimes \cdots \otimes f_n : f_i \in B\}.\]

And finally--our main focus--is 
\[ {\op B}^{\otimes n}:= \{ f_1 \otimes \cdots \otimes f_n :  f_1, \ldots, f_n \in {\op B}\}\]
 as defined in \S \ref{secn:bisets}.  We emphasize that this tensor power is relative to ${\op \cG}$, not $\cG$.
 
 \subsection*{Semigroup}
\label{subsec:semigroup} By taking disjoint unions, we obtain a semigroup 
\[ {\op B}^{\otimes *}:=\bigsqcup_{n=1}^\infty {\op B}^{\otimes n}\]
with a natural multiplicative grading by the degree of the map, which in addition has a natural structure of an ${\op \cG}$-biset. 

\subsection*{Tensor products versus composition} 
To keep the discussion more elementary, we focus on the $\cG$-bisets $B^{\circ n}$ and $B^{\otimes n}$, rather than on their opposites.  

There is a natural $\cG$-bi-equivariant map of $\cG$-bisets $B^{\otimes n} \to
B^{\circ n}$ given by 
\[ f_1\otimes f_{2}\otimes \cdots \otimes f_n \mapsto f_1 \circ \cdots \circ f_n.\]
While obviously surjective,  it can fail to be injective.  The following lemma will be used to present two examples
of this phenomenon. 

\begin{lem}\label{lemma:liftables} Let $f$ be a Thurston map for which
the map of $\cG$-bisets $B^{\otimes 2}\to B^2$ is injective, hence an isomorphism.  Then
  \begin{equation*}
\cG_{f^{\circ 2}}=\{a\in \cG_f : f@a\in \cG_f\}.
  \end{equation*}
\end{lem}
  \begin{proof} Extend $f$ to a basis $\sfX$ of $B$. By  \cite[Prop. 2.3.2]{nekrashevych:book:selfsimilar}, $\sfX \otimes \sfX$ is a basis of $B^{\otimes 2}$.  
  
   Let $a\in \cG_{f^{\circ 2}}$.  Then there exists
$b\in \cG$ such that $a \circ (f\circ f)= (f\circ f)\circ b$.  On the
other hand $a \circ f =f_1 \circ b_1$ for some $b_1\in \cG$ and some
$f_1\in \sfX$.  Likewise $b_1\circ f=f_2\circ b_2$ for
some $b_2\in \cG$ and $f_2\in \sfX$.  Hence
  \begin{equation*}
f \circ (f \circ b) =a \circ f \circ f =f_1\circ (f_2\circ b_2).
  \end{equation*}
Because $B^{\otimes 2}\to B^2$ is injective, this means that
  \begin{equation*}
(f \otimes f)\cdot b = f\otimes (f\circ b) = f_1 \otimes (f_2 \circ b_2)= (f_1 \otimes f_2)\cdot b_2.
  \end{equation*}
Since $\sfX^{\otimes 2}$ is a basis of $B^{\otimes 2}$ and the right action (!) is free, we have $f_1=f_2=f$ and $b=b_2$.  This and the equation $a \circ f =f_1 \circ b_1$ imply that $a\in \cG_f$ and $b_1=f@a$.  The equation $b_1 \circ f_2=f_1 \circ b_2$ even implies that $f@a \in \cG_f$.  Thus
  \begin{equation*}
\cG_{f^{\circ 2}}\subseteq \{a\in \cG_f:f@a\in \cG_f\}.
  \end{equation*}
The reverse inclusion is clear, and so the proof of the lemma is
complete. 

\end{proof}

\begin{ex}\label{ex:euclidean} \rm Let $f:S^2\to S^2$ be a Euclidean
NET map given by the matrix $A=\left[\begin{smallmatrix}0 & -1 \\ 2 &
0\end{smallmatrix}\right]$.  We identify $\cG$ with the subgroup of
$\text{PSL}(2,\Z)$ represented by matrices in the group $\Gamma(2)$
consisting of matrices which are congruent to the identity matrix modulo
2.  Then $\cG_f$ is the image in $\text{PSL}(2,\Z)$ of
$\Gamma(2)\cap (A\Gamma(2)A^{-1})$.  So $\cG_f$ does not contain the
image of $\left[\begin{smallmatrix}1 & 0 \\ 2 & 1
\end{smallmatrix}\right]$.  However, $f^{\circ 2}$ is represented by
the matrix $\left[\begin{smallmatrix}2 & 0 \\ 0 & 2
\end{smallmatrix}\right]$.  Clearly $\cG_{f^{\circ 2}}=\cG$.
Therefore $\cG_{f^{\circ 2}}\nsubseteq \cG_f$, and so
Lemma~\ref{lemma:liftables} implies that the map $B^{\otimes 2}\to
B^2$ is not injective.
\end{ex}

The previous example deals with Euclidean Thurston maps, making it
quite unsatisfying.  The next example shows that the same phenomenon
occurs for some maps with hyperbolic orbifolds.

\begin{ex}\label{ex:degree3} \rm Let $f$ be a Thurston map whose
postcritical set $P_f$ has exactly four points and whose homotopy type
is given by the finite subdivision rule in Figure~\ref{fig:fsrForf}.
\begin{figure}
\centerline{\scalebox{.6}{\includegraphics{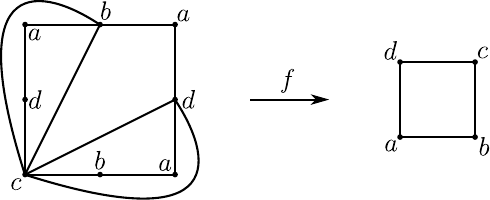}}}
\caption{A finite subdivision rule for the map $f$}
\label{fig:fsrForf}
  \end{figure}
We are using stereographic projection of the 2-sphere to the plane.
Every vertex label on the left indicates the image of the vertex in
the square pillowcase on the right.  The square on the left is the
same as the square on the right.

Let $\zg$ be a simple closed curve in $S^2-P_f$ with slope 2.  The
right half of Figure~\ref{fig:fsrForfPullback} shows two core arcs for
$\zg$, one blue joining $c$ and $d$ and one red joining $a$ and $b$.
The left half of Figure~\ref{fig:fsrForfPullback} shows the
$f$-preimages of these two core arcs.  Let $\zt$ be a primitive Dehn
twist about $\zg$.  \begin{figure}[h!]
\centerline{\scalebox{.5}{\includegraphics[width=4in]{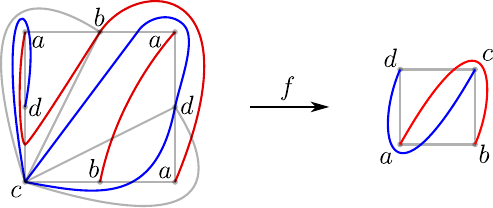}}}
\caption{The $f$-preimage of two core arcs for $\zg$}
\label{fig:fsrForfPullback}
  \end{figure}

We claim that $\zt^2\in \cG_f$.  To verify this, we view $\zt$ as
represented by a homeomorphism fixing pointwise a regular neighborhood
of $\zg$'s blue core arc and performing one rotation of a regular
neighborhood of $\zg$'s red core arc.  For the pullback, a regular
neighborhood of the blue connected component is fixed pointwise.
There are two red connected components, one mapping with degree 1 and
one mapping with degree 2.  The red connected component mapping with
degree 1 contains exactly one postcritical point, $c$.  Lifting
$\zt^2$ has the effect of inducing two complete rotations of a regular
neighborhood of this connected component.  These rotations can be
performed while fixing $c$.  Hence these rotations are homotopic to
the identity map in $S^2-P_f$.  For the red connected component
mapping with degree 2, the effect of $\zt^2$ is to perform one
complete rotation on a regular neighborhood.  This rotation represents
a Dehn twist about a simple closed curve with slope 1.  So $\zt^2\in
\cG_f$ and the image of $\zt^2$ under the pure modular group virtual
endomorphism is a primitive Dehn twist about a simple closed curve
with slope 1.  It is now easy to see that $\zt\notin \cG_f$, and so
$\zt^2$ generates the stabilizer of the homotopy class of $\zg$ in
$\cG_f$.

Now we consider the iterate $f^{\circ 2}$.  Figure~\ref{fig:fsrForfsq}
shows a finite subdivision rule for $f^{\circ 2}$.  
\begin{figure}
\centerline{\scalebox{.5}{\includegraphics[width=5in]{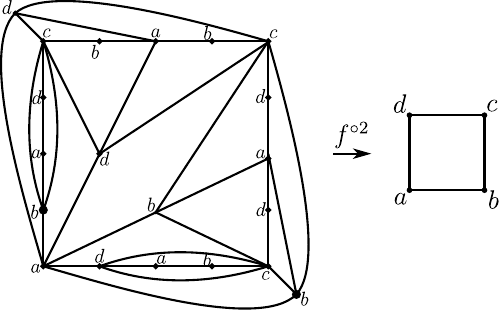}}}
\caption{A finite subdivision rule for the map $f^{\circ 2}$ }
\label{fig:fsrForfsq}
  \end{figure}

  \begin{figure}
\centerline{\scalebox{.4}{\includegraphics[width=7in]{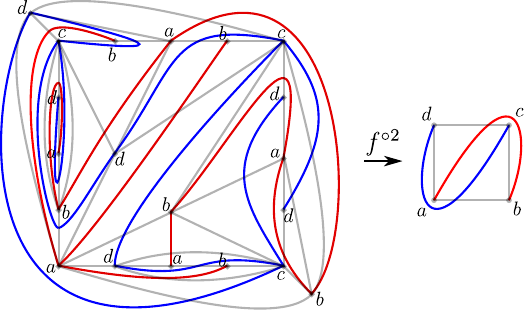}}}
\caption{The $f^{\circ 2}$-preimage of two core arcs for $\zg$  }
\label{fig:fsrForfsqPullback}
  \end{figure}

We claim that $\zt^3\in \cG_{f^{\circ 2}}$.  As for $f$, we view
$\zt$ as represented by a homeomorphism fixing pointwise a regular
neighborhood of the blue core arc in $S^2$ joining $c$ and $d$.  See
Figure~\ref{fig:fsrForfsqPullback}.  The $f^2$-preimage of the blue
arc is connected, and a regular neighborhood of it is fixed pointwise
when pulling back by $f^{\circ 2}$.  There are two red connected
components, one mapping with degree 3 and one mapping with degree 6.
The red connected component mapping with degree 3 contains exactly one
postcritical point, $a$.  Pulling back $\zt^3$ induces one complete
rotation of a regular neighborhood of this connected component.  This
rotation is homotopic to the identity map in $S^2-P_f$.  Pulling back
$\zt^3$ induces a half rotation of a regular neighborhood of the red
connected component mapping with degree 6.  This red connected
component contains no postcritical points, and this half rotation is
homotopic to the identity map in $S^2-P_f$.  Hence $\zt^3\in
\cG_{f^{\circ 2}}$ and $\zt^3$ lifts to the identity element of
$\cG$.

So we have that $\zt^2$ generates the stabilizer of the homotopy class
of $\zg$ in $\cG_f$ and that $\zt^3\in \cG_{f^{\circ 2}}$.  It
follows that $\cG_{f^{\circ 2}}$ is not contained in $\cG_f$.
Therefore Lemma~\ref{lemma:liftables} implies that the map $B^{\otimes2
}\to B^2$ is not injective.

\end{ex}

\subsection{Examples of Hurwitz bisets}
\label{subsec:examples}

We conclude this section with more examples.

\begin{ex}
\label{ex:rnc}Composition of unobstructed can be obstructed.  {\rm See \cite[\S 6]{bartholdi:nekrashevych:twisted} for details.  For consistency with the presentation there, we work in the semigroup ${\op B}^{\circ *}$.  Let $f_i(z)=z^2+i$, let $P:=P_f:=\{i, -1+i, -i, \infty\}$ denote the postcritical set of $f_i$, let ${\op B}$ be the ${\op \cG}$-Hurwitz biset determined by $f_i$, and $({\op B})^{\circ *}$ the corresponding semigroup.  Let $a$ denote the right Dehn twist about the boundary of a regular neighborhood of the Euclidean segment joining $i$ and $-i$.  Then since $a^2$ lifts under $f_i$ to the trivial element, we have $f_i \cdot a^2 = f_i$ in ${\op B}^{\circ *}$.  The map $f_i\cdot a$ is obstructed.  Since in ${\op B}^{\circ *}$ we also have 
\[ (f_i\cdot a) \circ (f_i\cdot a) = f_i \circ (a \cdot f_i \cdot
a^2 \cdot a^{-1}) = f_i \circ (a \cdot \text{id} \cdot f_i \cdot
a^{-1}) = f_i \circ (a \cdot f_i \cdot a^{-1}),\]
we see that in ${\op B}^{\circ *}$ the obstructed map $(f_i \cdot
a)\circ (f_i \cdot a)$ coincides with the composition of the unobstructed maps $f_i$ and $a^{-1}\cdot f_i \cdot a$.}
\end{ex}

A Hurwitz biset might have finitely many (even zero) or infinitely
many obstructed or unobstructed Thurston equivalence classes.  The
remaining examples illustrate various possibilities.  By the folklore
result in Corollary 3.5 of \cite{MR1749249}, every Hurwitz biset
contains only finitely many rational Thurston equivalence classes.

\begin{ex}\label{ex:finrat} Finitely many rational classes, no
obstructed classes. 
\begin{enumerate}
  \item The Hurwitz biset of the rabbit polynomial.  \emph{See
\cite{bartholdi:nekrashevych:twisted}.}
  \item Polynomials for which every cycle in the postcritical set
contains a critical point.  \emph{This is the Levy-Berstein theorem
\cite[Theorem 10.3.9]{MR3675959}.}
  \item Thurston maps with contracting mapping class group bisets.
\emph{This follows from Theorem~\ref{thm:joint}.}
  \item Certain critically fixed Thurston maps.  \emph{ See
\cite{cgnpp} and \cite{HP}.  For $d \geq 2$, suppose $2d-2=m_1+\cdots
+ m_n$ is a partition, with $m_i \leq d-1$ for each $i$, and $n \leq
d$. By \cite[Thm. 1.2]{cgnpp}, there is a connected planar multigraph
with $n$ vertices of valences $m_1, \ldots, m_n$. Blowing up the edges
of this multigraph gives a critically fixed Thurston map $f$ that is
equivalent to a rational map. The Hurwitz biset $B:=B(f)$ depends only
on the partition.  If every such graph realizing the given valences is
connected, then $B$ consists entirely of rational maps.  Given that
$2d-2=m_1+\cdots +m_n$, the condition that $m_i\le d-1$ is equivalent
to the condition that $m_i\le\sum_{j\ne i}m_j$.  So every such map is
connected if and only if the multiset $\{m_1, \ldots, m_n\}$ cannot be
split into two disjoint nonempty subsets $A_1, A_2$ such that for each
$i=1,2$, no element of $A_i$ strictly exceeds the sum of the other
elements.  For example, $2\cdot 6 - 2 = 4+2+2+1+1$ and $2\cdot 5-2 =
3+2+2+1$ each give non-polynomial examples.}

%One can show that if the multiset $\{m_1, \ldots, m_n\}$ cannot be split into two disjoint nonempty subsets $A_1, A_2$ such that for each $i=1,2$, no element of $A_i$ strictly exceeds the sum of the other elements, then any such graph is connected. The example $2\cdot 3-2 = (2+2)+(2+2)$ shows such a condition is necessary. }

\end{enumerate}
\end{ex}

\begin{ex}\label{ex:finratfinnonrat} A mixture of finitely many
rational and finitely many obstructed classes. \emph{Here is a general
procedure for constructing such bisets.  Let $g\co (S^2,P_g)\to
(S^2,P_g)$ be a rational Thurston map with postcritical set $P_g$ of
order 4 and hyperbolic orbifold for which there exists a simple closed
curve $\zg$ in $S^2-P_g$ which is neither peripheral nor null
homotopic such that $g^{-1}(\zg)$ contains a connected component in
the same pure mapping class group orbit as $\zg$ (the strata scrambler
of $g$ contains a loop).  Let $h\co (S^2,P_g)\to (S^2,P_g)$ be a
Latt\`{e}s map given by a matrix $\left[\begin{smallmatrix}a & b \\ c
& -a\end{smallmatrix}\right]$, where $a$, $b$ and $c$ are integers
with $a$ odd and $b$, $c$ both even.  We may choose $h$ so that
$p:=\deg(h)$ is prime, necessarily odd.  Let $f=h\circ g$.  Every
multiplier for $f$ is a multiplier for $h$ times a multiplier for $g$.
Every multiplier of $h$ is either $p$ or $1/p$ and both occur.  So it
is easy to choose $h$ so that $f$ has a nonzero multiplier less than
1, a multiplier greater than 1 but no multiplier equal to 1.  Because
no multiplier equals 1, Theorem 5 in \cite{kmp:origami} implies that
the Hurwitz biset containing $f$ has only finitely many Thurston
equivalence classes.  Because the matrix of $h$ is congruent to the
identity modulo 2, the pullback map of $h$ on simple closed curves
takes every simple closed curve to a simple closed curve in the same
pure mapping class group orbit.  So the strata scrambler of $f$ has a
loop at some vertex.  Because multipliers of $h$ depend on congruences
modulo $p$ and edges of strata scramblers depend on congruences modulo
2, this loop for $f$ has a multiplier greater than 1.  It follows that
the Hurwitz biset of $f$ has an obstructed element.}
\end{ex}

\begin{ex}\label{ex:finratinfnonrat} A mixture of finitely many
rational and infinitely many obstructed classes.
\begin{enumerate}
  \item $z\mapsto z^2+i$. \emph{See \ref{ex:dendrite} and
\cite{bartholdi:nekrashevych:twisted}.}
  \item Critically fixed Thurston maps $f$ with at least four
postcritical points, with at most $\deg(f)$ postcritical points and
whose associated partition multiset can be realized by a disconnected
graph.  \emph{See (4) in Example~\ref{ex:finrat}.}
\end{enumerate} 
\end{ex}

\begin{ex}\label{ex:infnonrat} Infinitely many obstructed classes. 
\begin{enumerate}
  \item Basilica mated with basilica.
  \item Critically fixed Thurston maps $f$ with more postcritical points than
$\deg(f)$.  \emph{This follows from Theorem 1.1 in \cite{cgnpp}.}
\end{enumerate}
\end{ex}

\section{Strata scrambler and IFS on $\cV$} 
\label{secn:scrambler}

\subsection{Curves, multicurves, lifts, and Thurston linear maps} Fix
a finite subset $P \subset S^2$ with $\#P \geq 3$.  We denote by
$\curves$ the set of isotopy classes of simple, closed, unoriented
curves in $S^2-P$.  We call such a class a ``curve'', for simplicity.
A curve which is inessential or peripheral we call \emph{trivial}.
Thus e.g. when $\#P=3$ we have $\#\curves=4$.

A \emph{multicurve} is a possibly empty collection of curves
$\Gamma=\{\gamma_1, \ldots, \gamma_m\} \subset \curves$ such that each
$\gamma_i$ is nontrivial, $\gamma_i \neq \gamma_j$ for $i \neq j$
(that is, representatives of the $\gamma_i$'s are pairwise
non-isotopic), and the collection $\Gamma$ is simultaneously
representable by pairwise disjoint elements.  We denote by $\mc$ the
set of multicurves.

The group $\cG$ acts on $\mc$ via the right pullback action with
finitely many orbits.  So ${\op \cG}$ acts on the left by pullback with
finitely many orbits:
\[ (g\cdot h).\Gamma= (g\cdot h)^{-1}(\Gamma)=(h^{-1}\cdot
g^{-1})(\Gamma)=g^{-1}\circ
h^{-1}(\Gamma)=g^{-1}(h^{-1}(\Gamma))=g.(h.\Gamma).\]
The empty multicurve $\emptyset$ is a globally fixed element. We
denote by ${\op \cG}.\Gamma$ the orbit of $\Gamma$ under ${\op \cG}$.

Let $B$ be a Hurwitz class of Thurston maps relative to $P$.  An
element $f \in B$ induces, via pullback, an
at-most-one-to-$\deg(f)$-relation on $\curves$: for $\gamma,
\widetilde{\gamma} \in \curves$, we denote by $\gamma \leftarrow
\widetilde{\gamma}$ whenever a component of the $f$-pullback of a
representative of $\gamma$ represents $\widetilde{\gamma}$.  We let
$f^{-1}(\gamma)$ denote the set of all such curves
$\widetilde{\gamma}$.  Analogously, there is a well-defined function
on multicurves induced by pullback, which we write as $\Gamma
\leftarrow f^{-1}(\Gamma)$.  More precisely, \[ f^{-1}(\Gamma) :=
\{\widetilde{\gamma}\; \text{nontrivial}: \exists \gamma \in \Gamma,
\gamma \leftarrow \widetilde{\gamma}\}.\] In particular: (i)
$f^{-1}(\emptyset) = \emptyset$, i.e. the empty multicurve pulls back
to itself, (ii) if every preimage of every element of $\Gamma$ is
trivial, then $f^{-1}(\Gamma)=\emptyset$, and (iii) we do not allow
trivial curves to be elements of $f^{-1}(\Gamma)$.  Note that the
pullback relation on curves allows trivial curves in the inverse
image, while the pullback function on multicurves does not. This is
for technical convenience.  A nonempty multicurve is called
\emph{completely invariant} if $f^{-1}(\Gamma)=\Gamma$.

For a nontrivial $\Gamma \in \mc$, we denote by $\R[\Gamma]$ the free
$\R$-module spanned by $\Gamma$; in particular, this is a real vector
space of dimension $\#\Gamma$ equipped with the basis $\Gamma$.  We
set $\R[\Gamma]=\{0\}$, the zero-dimensional real vector space, if
$\Gamma$ is the empty multicurve.  Note that since $\cG$ is the pure
mapping class group, if $g \in {\op \cG}$ stabilizes $\Gamma \in \mc $
setwise, then $g.\gamma = \gamma$ for each $\gamma \in \Gamma$. It
follows that if ${\op \cG}.\Gamma_1={\op \cG}.\Gamma_2$ then the
vector spaces $\R[\Gamma_1]$ and $\R[\Gamma_2]$ are canonically
identified.  If $f \in B$ and $\Gamma \in \mc(S^2,P)$ is nontrivial,
then there is an induced linear map 
\[ L[f,\Gamma]: \R[\Gamma]\to\R[f^{-1}(\Gamma)]\] 
defined on basis elements $\gamma\in \Gamma$ by
choosing a representative $\delta\in \gamma$ and setting 
\[L[f,\Gamma](\gamma)= \sum_{\substack{\gamma \leftarrow
\widetilde{\gamma}\\ \widetilde{\gamma}\; \text{nontrivial}}}\;
\sum_{\substack{\widetilde{\delta} \subset f^{-1}(\delta)\\
\widetilde{\delta} \in\widetilde{\gamma}}}
\frac{1}{\deg(f:\widetilde{\delta} \to \delta)}\widetilde{\gamma}.\]
If $\Gamma=\emptyset$ we let $L[f,\Gamma]: \{0\} \to \{0\}$ be the
zero map.  Note that $f^{-1}(\Gamma)=\emptyset$ if and only if
$L[f,\Gamma] \equiv 0$.

\subsection{Linear iterated function system on a vector space}
\label{subsecn:ifs}
Our goal is to define a finite invariant of $B$ which organizes all
such linear maps into a dynamical system that respects composition in
the semigroup $B^*$.  Informally, we proceed as follows.   We begin
with the finite directed graph with vertex set ${\op \cG}\backslash \mc$ and a
directed edge ${\op \cG}.\Gamma\rightarrow {\op \cG}.\widetilde{\Gamma}$
 if and only if for some $\Gamma' \in
{\op \cG}.\Gamma$  we have $f^{-1}(\Gamma' )\in
{\op \cG}.\widetilde{\Gamma}$; this edge is equipped with a weight--in the
form of a linear map between finite-dimensional $\R$-vector
spaces--given by $L[f,\Gamma']:\R[\Gamma'] \to \R[f^{-1}(\Gamma')]$.

This almost works--except that the domains and codomains are drawn from an infinite set.  We resolve this using the canonical identifications. A by-product of the analysis is that this invariant of $B$ may in principle be computed in finite time. 

Let $\sfX$ be a basis for the biset $B$; $t$o emphasize the fact that
elements of $\sfX$ are maps, we denote the elements of $\sfX$ by $f_{\sfx}$.  Let $\sfT$ denote an orbit transversal for the action of ${\op \cG}$ on $ \mc$; we similarly denote its elements by $\Gamma_{\sft}$. Given $\Gamma_{\sft} \in \sfT$ and $f_{\sfx} \in \sfX$, there exists a unique $\Gamma_{\tilde{\sft}} \in \sfT$ and a non-unique $h=h(f_\sfx, \Gamma_{\sft}) \in {\op \cG}$ with $f_{\sfx}^{-1}(\Gamma_{\sft}) =h(\Gamma_{\tilde{\sft}})$.  For each pair $(f_\sfx, \Gamma_{\sft}) \in \sfX \times \sfT$, we choose and fix such an $h=h(f_\sfx, \Gamma_{\sft})$. The induced linear map $L[f_{\sfx}\circ h,\Gamma_{\sft}] :\R[\Gamma_{\sft}] \to \R[\Gamma_{\tilde{\sft}}]$ is independent of the choice of $h$.  To resolve the finiteness issue, we create a directed weighted edge from $\Gamma_{\sft}$ to $\Gamma_{\tilde{\sft}}$ only for ordered pairs $(f_\sfx, \Gamma_{\sft})$ in the finite set $\sfX \times \sfT$ via the procedure just described. 

Let us now see that we do not lose any information with this restriction.  Fix $f \in B$ and $\Gamma \in  \mc$.  Consider next the diagram below:
 \[
 \xymatrix{
(S^2, P, f^{-1}(\Gamma)) \ar[d]_f \ar[r]_{k}  \ar@/^1pc/[rr]^{\tilde{g}}&(S^2, P, f_{\sfx}^{-1}(\Gamma_{\sft})) \ar[d]^{f_{\sfx}} \ar[r]_(.57){h^{-1}}&(S^2, P, \Gamma_{\tilde{\sft}}) \ar[d]^{f_{\sfx}\circ h}\\
(S^2, P, \Gamma) \ar[r]^{g} \ar@/_1pc/[rr]_{g}&(S^2, P, \Gamma_{\sft}) \ar[r]^{\text{id}}& (S^2, P, \Gamma_{\sft})
}
 \]
Since $\sfT$ is an orbit transversal, there exists a (non-unique) $g \in {\op \cG}$ and unique $\Gamma_{\sft} \in \sfT$ with $g(\Gamma)=\Gamma_{\sft}$.  Since $B$ is a covering biset and $\sfX$ is a basis, there are unique $k \in {\op \cG}$ and $f_{\sfx} \in \sfX$ such that $g \circ f = f_{\sfx} \circ k$ in $B$, so the left-hand square commutes. The right-hand square commutes by the previous paragraph. So since the identifications of the vector spaces spanned by multicurves are canonical, the finite set of linear maps $L[f_{\sfx}\circ h,\Gamma_{\sft}]: \R[\Gamma_{\sft}] \to \R[\Gamma_{\tilde{\sft}}]$ for $f_\sfx \in \sfX$ and $\Gamma_{\sft} \in \sfT$ coincides with the set of linear maps induced by arbitrary elements of $B$ on arbitrary multicurves, once the domains and codomains are so identified.

The multicurves $\Gamma_{\sft}$ and $\Gamma_{\tilde{\sft}}$ are independent of
the choice of $g$ and of $h$, so the vector spaces $\R[\Gamma_{\sft}]$ and
$\R[\Gamma_{\tilde{\sft}}]$ are independent of these choices. The fact
that the right-hand square above commutes as maps on the sphere
implies that the linear maps $L[\widetilde{f}\circ h,\Gamma]: \R[\Gamma]
\to \R[\Gamma_{\tilde{\sft}}]$ are also independent of these choices. 

We have just shown for every $f\in B$ and $\Gamma\in \mc$ that there
exist $f_\sfx\in \sfX$, $\Gamma_{\sft}\in \sfT$ such that
$L[f,\Gamma]:\R[\Gamma]\to \R[f^{-1}(\Gamma)]$ is equivalent to
$L[f_{\sfx}\circ h,\Gamma_{\sft}]: \R[\Gamma_{\sft}] \to
\R[\Gamma_{\tilde{\sft}}]$ for some $f_\sfx \in \sfX$,
$\Gamma_{\sft},\Gamma_{\tilde{\sft}} \in \sfT$ and $h=h(f_\sfx,\Gamma_{\sft})$.
Similar reasoning shows that instead of fixing $f$ and $\Gamma$ and
finding $f_\sfx$ and $\zG_{\sft}$, we can fix $f$, $f_\sfx$, $\zG_{\sft}$ and
find $\Gamma$.  Thus for every $f\in B$ the set of linear maps
$L[f_{\sfx}\circ h,\Gamma_{\sft}]: \R[\Gamma_{\sft}] \to
\R[\Gamma_{\tilde{\sft}}]$ such that $f_\sfx \in \sfX$ and
$\Gamma_{\sft},\Gamma_{\tilde{\sft}} \in \sfT$ is equivalent to the set of maps
$L[f,\Gamma]:\R[\Gamma]\to \R[f^{-1}(\Gamma)]$ for $\Gamma\in \mc$.

We obtain a collection of linear maps on a single space as follows. We
set \[ \cV:=\bigoplus_{\Gamma_{\sft} \in \sfT}\R[\Gamma_{\sft}].\] We extend each
of the linear maps $L[f_{\sfx}\circ h,\Gamma_{\sft}]: \R[\Gamma_{\sft}] \to
\R[\Gamma_{\tilde{\sft}}]$ to a linear map $a: \cV \to \cV$ as follows.
We set $a \equiv 0$ on each summand other than that determined by
$\Gamma_{\sft}$.  On $\Gamma_{\sft}$, we set $a$ to be $L[f_{\sfx}\circ
h,\Gamma_{\sft}]: \R[\Gamma_{\sft}] \to \R[\Gamma_{\tilde{\sft}}]$, followed by the
natural inclusion $R[\Gamma_{\tilde{\sft}}]\hookrightarrow \cV$. We call
the summand $\R[\Gamma_{\sft}]$ the \emph{support} of $a$.  Thus each $f
\in B$ induces a finite set of elements \[ L[B]:=\{ \;[ L[f,\Gamma] ]:
\Gamma \in \mc\} \subset \End(\cV), \] where here the outside square
brackets denote the element of $\End(\cV)$ induced by $L[f,\Gamma]$.
The previous paragraph shows that this set is independent of the
choice of $f$.

In order to define matrices, we now choose an ordered basis for $\cV$
as follows.  We arbitrarily choose a linear ordering on the set of
${\op \cG}$-orbits of curves.  Let $\zG$ be a multicurve.  As noted
earlier, since $\cG$ is the pure mapping class group, if $g \in {\op \cG}$
stabilizes $\Gamma$ setwise, then $g.\gamma = \gamma$ for each $\gamma
\in \Gamma$.  It follows that we obtain a ${\op \cG}$-equivariant linear
ordering of the elements of $\Gamma$.  This equips $\R[\Gamma]$ with a
distinguished ordered basis.

The multicurve $\zG$ determines a unique $\zG_{\sft}\in \sfT$ and therefore a
unique summand $\R[\Gamma_{\sft}]$ of $\cV$.   
The distinguished ordered bases of the summands of
$\cV$ combine to form a distinguished ordered basis of $\cV$ after
lexicographically ordering the summands.  This ordered basis of $\cV$
allows us to speak of the matrix of an element of $\text{End}(\cV)$.

We remark that the definitions of the pullback function $\Gamma
\mapsto f^{-1}(\Gamma)$ on multicurves and of the linear map
$L[f,\Gamma]$ imply that given $f$ and $\Gamma$, the matrix for
$L[f,\Gamma]: \R[\Gamma] \to \R[f^{-1}(\Gamma)]$--if it exists--has no
zero rows.  A linear map with codomain a zero-dimensional vector space
does not have a matrix!  (Of course, the matrix of
$[L[f,\Gamma]]:\cV\to \cV$ has many zero rows in general.)

In case $f^{-1}(\Gamma') \neq \emptyset$, we can represent the weight
by the matrix of this linear map; otherwise, we represent it by the
symbol $0$, without ``matrix brackets''.

As a dynamical system, this construction is natural:

\begin{prop}
\label{prop:choices}
Suppose $\sfT, \sfT'$ are two choices of orbit transversals to the action of $\cG$ on $\mc$, and $\sfX, \sfX'$ are two bases for the Hurwitz biset $B$. 
Let $\cV, \cV'$ be the corresponding vector spaces, and suppose $f \in B$. We denote by $L[B]=\{[L[f]]\}, L'[B]=\{[L'[f]]\}$ the sets of elements of $\text{End}(\cV), \text{End}(\cV')$, respectively, determined by these choices. 
\begin{enumerate}
\item There is a canonical isomorphism $\Phi: \cV \to \cV'$ induced by $\Gamma_{\sft} \mapsto \Gamma_{\sft'}$ where ${\op \cG}.\Gamma_{\sft} = {\op \cG}.\Gamma_{\sft'}$. 
\item For each $a \in \{[L[f]]\}$, the corresponding $a' \in \{[L'[f]]\}$ is given by $a'=\Phi \circ a \circ \Phi^{-1}$.
\item $L'[B]=\Phi \circ L[B] \circ \Phi^{-1}$, as subsets.
\end{enumerate}

\end{prop}

\begin{proof} The first statement is obvious. Denoting with primes the objects arising in the corresponding construction using $\sfT'$ and $\sfX'$, the second follows from considering the diagram 
\[
 \xymatrix{
(S^2, P, \Gamma_{\tilde{t}'}) \ar[d]_{f_{\sfx'}\circ h'}&\ar[l]_(.53){(h')^{-1} \circ k'} (S^2, P, f^{-1}(\Gamma)) \ar[d]_{f}\ar[r]^(.56){h^{-1} \circ k}&  (S^2, P, \Gamma_{\tilde{\sft}})\ar[d]^{f_{\sfx}\circ h} \\
(S^2, P, \Gamma_{t'}) & \ar[l]_(.45){g'} (S^2, P, \Gamma) \ar[r]^{g} & (S^2, P, \Gamma_{\sft})
}
 \]
and recalling that the induced linear maps are independent of the
choices. The third conclusion follows from the second, by taking
unions. 
 
\end{proof}

\begin{defn} \label{defn:scrambler} The \emph{strata scrambler} of
$B$, denoted $\cS[B]$, is the directed weighted graph whose vertices
are the vector spaces $\R[\Gamma_{\sft}], \Gamma_{\sft} \in \sfT$, and with an edge
from $\R[\Gamma_{\sft}]$ to $\R[\Gamma_{\tilde{\sft}}]$ weighted by a linear
map $a: \R[\Gamma_{\sft}] \to \R[\Gamma_{\tilde{\sft}}]$ if and only if for
some $f \in B$ and $\Gamma \in \mc(S^2,P)$ we have $\Gamma_{\sft}\in
{\op \cG}.\Gamma$, $f^{-1}(\Gamma)\in {\op \cG}.\Gamma_{\tilde{\sft}}$, and
$L[f,\Gamma]$ induces $a$ as above.  We allow loops and multiple edges
between two vertices, but the weights of edges between two given
vertices are distinct.  
\end{defn}

We proceed with some important remarks:
\begin{enumerate}

\item Note that the trivial endomorphism $\cV \to \{0\}$ belongs to $L[B]$, since $f^{-1}(\emptyset)=\emptyset$ for any $f \in B$. 

\item In the definition, we emphasize that if $\Gamma_1, \Gamma_2$ are
nontrivial multicurves with $\Gamma_1 \subset \Gamma_2$ but $\Gamma_1
\neq \Gamma_2$, then ${\op \cG}.\Gamma_1 \neq {\op \cG}.\Gamma_2$ and so the
corresponding subspaces of $\cV$ intersect only at the origin;
e.g. $\R[\Gamma_1] \not\subset \R[\Gamma_2]$.
\end{enumerate}

\subsection{Basic properties of $L[B]$}
\label{subsecn:basicL}

The following proposition follows immediately from the definitions.

\begin{prop} \label{prop:whenzero} Suppose that $a_1,\dotsc,a_n \in
L[B]$ for some positive integer $n$ and that $a_i$ is induced by
$L[f_i,\Gamma_i]: \R[\Gamma_i] \to \R[\Gamma_i']$, where $\Gamma_i' =
f_i^{-1}(\Gamma_i)$.
\begin{enumerate} 
  \item If ${\op \cG}.\Gamma_i'\neq {\op \cG}.\Gamma_{i+1}$ for
some $i\in \{1,\dotsc,n-1\}$, then $a_n\circ \cdots \circ a_2\circ a_1
=0$.
  \item Suppose that ${\op \cG}.\Gamma_i'={\op \cG}.\Gamma_{i+1}$ for $i\in
\{1,\dotsc,n-1\}$.  Then there exists $g_i \in {\op \cG}$ such that
$g_i(\Gamma_{i+1})=\Gamma_i'$ for $i\in \{1,\dotsc,n-1\}$, and for the
map $F:=f_n\cdot g_{n-1}\cdot \cdots \cdot f_2\cdot g_1\cdot f_1 \in
B^n$, we have $L[F,\Gamma_1]: \R[\Gamma_1] \to \R[\Gamma_n']$ induces
the composition $a_n\circ \cdots \circ a_2\circ a_1$ in $\text{End}(\cV)$.
  \item If ${\op \cG}.\Gamma_1 = {\op \cG}.\Gamma_1'$, then up to
permutation of coordinates, the matrix for $a_1$ is given by 
\[ \left[
\begin{array}{cc}
A & 0 \\
0 & 0 
\end{array}
\right]
\]
where $A$ is the matrix for $L[f_1,\Gamma_1]: \R[\Gamma_1] \to \R[\Gamma'_1]$. 
\end{enumerate}
\end{prop}

Statement 2 of Proposition~\ref{prop:whenzero} implies that every
directed edge path in $\cS[B]$ induces an element of
$\text{End}(\cV)$.  In this way the strata scrambler determines an
iterated function system on $\cV$.  In fact, the assignment $B \mapsto
L[B]$ is functorial: for each $n\in \N$ we have \[ L[B^{\otimes n}] =
(L[B])^n,\] where \[ L[B^{\otimes n}] = \bigcup_{f_n, \ldots, f_1 \in
B} \{L[f_n\cdot \cdots \cdot f_1]\} \subset \text{End}(\cV)\] and \[
(L[B])^n = \{a_n\circ \cdots \circ a_{2} \circ a_1 : a_i \in L[B]\}
\subset \text{End}(\cV).\] This leads to
  \begin{equation*}
L[B^{\otimes *}]:=\bigcup_{n=1}^{\infty }L[B^{\otimes
n}]=\bigcup_{n=1}^{\infty }(L[B])^n=:L[B^*].
  \end{equation*}
This set is a semigroup for the operation of composition in
$\text{End}(\cV)$.

\subsection{Relating the iterated function systems} 
\label{subsecn:relating}

We wish to relate the IFS on ${\op \cG}$ consisting of operators $g\mapsto
\sfx@g$ with the IFS on $\cV$ given by the strata scrambler $\cS[B]$.

We begin with a multicurve $\zG$.  Every $\zg\in \zG$ determines a
positive (right) Dehn twist $\twist(\gamma)$ about $\gamma$.  Let
$\text{Tw}(\zG)$ denote the multitwist subgroup of $\cG$ generated by
the Dehn twists $\text{tw}(\zg)$ for $\zg\in \zG$.  As noted earlier,
since $\cG$ is the pure mapping class group, if $g \in {\op \cG}$
stabilizes $\Gamma$ setwise, then $g.\gamma = \gamma$ for each $\gamma
\in \Gamma$.  It follows that we obtain a canonical bijection from
$\zG$ to $\zG_{\sft}$ for some $\Gamma_{\sft}\in \sfT$ and therefore a
canonical isomorphism $\zi_\zG\co \text{Tw}(\zG)\to
\Z[\zG_{\sft}]\subset \cV_\Z$ such that $\zi_\zG(\text{tw}(\zg))=1\cdot
\zg_{\sft}$, where $\zg_{\sft} \in \zG_{\sft}$ is the element of
$\zG_{\sft}$ corresponding to $\zg\in \zG $.

Now let $f\in B$.  We choose the biset basis $\sfX\subseteq B$ to
include $f$.  The operator $f@\co {\op \cG}\to {\op \cG}$ is defined
so that if $g\in {\op \cG}$, then $f@g$ is the element $k\in {\op
\cG}$ such that $f\cdot g=k\cdot f'$ for some $f'\in \sfX$.  Recall
that ${\op \cG}_f=\text{Stab}_{{\op \cG}}{\op \cG}\cdot f$.  In other
words, $g\in {\op \cG}_f$ if and only if $f'=f$.  It follows that the
restriction of $f@ $ to ${\op \cG}_f$ is a group homomorphism.
Combining i) the equation $f\cdot g=k\cdot f$ and ii) the definition
of $L[f,\zG]$ and iii) the definition of Dehn twist and iv) the
canonical identifications of $\R[\zG]$ and $\R[f^{-1}(\zG)]$ with
subspaces of $\cV$, we obtain the following proposition.

\begin{prop}\label{prop:AtL} Let $f\in B$, and let $\zG$ be a
multicurve.  Let $\zG_{\sft}$ and $\zG_{\widetilde{t}}$ represent the
${\op \cG}$-orbits ${\op \cG}.\zG$ and ${\op \cG}.f^{-1}(\zG)$, respectively.  Then
the following diagram commutes.
  \begin{equation*}
 \xymatrix{
\text{Tw}(\zG)\cap {\op \cG}_f\ar[d]^{\zi_\zG} \ar[r]^(.52){f@} &
\text{Tw}(f^{-1}(\zG))\ar[d]^{\zi_{f^{-1}(\zG)}}\\
\R[\zG_{\sft}]\ar[r]^{[L[f,\zG]]} &
\R[\zG_{\widetilde{t}}]
}
  \end{equation*}
\end{prop}

We continue in this vein.  The following paragraphs lead to
Proposition~\ref{prop:gtilde}.

Let $\sfX$ be a basis of $B$.  For each pair $(f_\sfx, \Gamma_{\sft})
\in \sfX \times \sfT$, the intersection $\Tw(\Gamma_{\sft})\cap {\op
\cG}_{f_\sfx}$ is a subgroup of $\Tw(\Gamma_{\sft})$ of finite index.
Every finite index subgroup of every finitely generated free Abelian
group $\bbZ^m$ contains elements of the form $(k,k,k,\dotsc,k)$ for
arbitrarily large integers $k$.  So every coset of such a subgroup
contains elements whose coordinates are all positive.  Thus for every
$\zG_{\sft}\in \sfT$ we can choose a finite set of such coset
representatives whose elements are products of powers of positive Dehn
twists about every one of the elements of $\Gamma_{\sft}$.  It might
be helpful for intuition to note that the $\ell^1$-norms of these
elements can be bounded in terms of the group index and
$\#\Gamma_{\sft}$--but we do not need this fact; only the finiteness
will be used. Doing this for each such pair $(f_\sfx, \Gamma_{\sft})
\in \sfX \times \sfT$, we obtain a finite collection $E \subset {\op
\cG}$ of such coset representatives.

Let $(f_\sfx,\zG_{\sft})\in \sfX\times \sfT$.  We obtain a map
$L[f_\sfx,\zG_{\sft}]\co \R[\zG_{\sft}]\to
\R[f_\sfx^{-1}(\zG_{\sft})]$.  Recall that we have fixed $h \in {\op
\cG}$, with $f_\sfx^{-1}(\Gamma_{\sft})=h(\Gamma_{\tilde{\sft}})$ and
$\zG_{\widetilde{t}}\in \sfT$.  We set $H$ to be the finite collection
of such elements $h$.

We next define a map $g\mapsto \widetilde{g}$ from
$\text{Tw}(\zG_{\sft})$ to $\text{Tw}(\zG_{\widetilde{t}})$ as
follows.  Let $g\in \text{Tw}(\zG_{\sft})$.  Then there exists $k\in
\text{Tw}(\zG_{\sft})\cap \cG_{f_x}\op$ and $e\in E$ such that $g =
k\cdot e$.  So $f_\sfx\cdot k=\widetilde{g}'\cdot f_\sfx$ for some
$\widetilde{g}'\in \text{Tw}(\zG_{\widetilde{\sft}})$.  Let
$f_\sfy:=f_\sfx\cdot e\in \sfX\cdot E$.  Consider the diagram below:
  \begin{equation}\label{eqn:gtilde}
\begin{gathered}
\xymatrix{
(S^2, P, \Gamma_{\widetilde{t}}) \ar[r]^{\widetilde{g}}  
\ar[r]\ar@/^1.4pc/[rr]^{\widetilde{g}} \ar[d]_{h} & 
(S^2, P, \Gamma_{\widetilde{t}}) \ar[r]^{\text{id}} \ar[d]_{h} & 
(S^2, P, \Gamma_{\widetilde{t}}) \ar[d]_{h} \\
(S^2, P, f_{\sfx}^{-1}(\zG_{\sft})) \ar[r]^{\widetilde{g}'} \ar[d]_{f_{\sfx}} & 
(S^2, P, f_{\sfx}^{-1}(\zG_{\sft})) \ar[r]^{\text{id}} \ar[d]_{f_{\sfx}} & 
(S^2, P, f_{\sfx}^{-1}(\zG_{\sft})) \ar[d]_{f_{\sfy}} \\
(S^2, P, \Gamma_{\sft}) \ar[r]^{k} \ar[r]\ar@/_1.2pc/[rr]_{g} & 
(S^2, P, \Gamma_{\sft}) \ar[r]^{e} & 
(S^2, P, \Gamma_{\sft}).
}
\end{gathered}
  \end{equation}

This diagram determines a map $g\mapsto \widetilde{g}$ from
$\text{Tw}(\zG_{\sft})$ to $\text{Tw}(\zG_{\widetilde{t}})$.  We see that
  \begin{equation*}
\widetilde{g}' = f_\sfx @ k = f_\sfx @ (g\cdot e^{-1}) =
(f_\sfx @ g) \cdot (f_\sfz@ e^{-1})
  \end{equation*}
for some $f_z\in \sfX$, and so 
\[
\begin{array}{lcr}
\widetilde{g}& = & h\cdot (f_\sfx@g) \cdot
((f_\sfz@e^{-1})\cdot  (h^{-1})).
\end{array}
\] 
The factors $h$ and $(f_\sfz@e^{-1})\cdot h^{-1}$ are each drawn
from finite sets $H$ and $\sfX(E^{-1})\cdot~\!\!H^{-1}$. Putting $C:=H
\cup (\sfX(E^{-1})\cdot H^{-1})$, we conclude that
  \begin{equation*}
\widetilde{g} = c_1\cdot (f_\sfx @ g) \cdot c_2 \text{ with } 
\;\; c_1, c_2 \in C.
  \end{equation*}
On the other hand, using the shorthand notation $\vec{r}$ for terms of
the form $\zi_{\zG_{\sft}}(r)$, Proposition~\ref{prop:AtL} implies that
\[\vec{\widetilde{g}}=[L[f_\sfx, \Gamma_{\sft}]]\cdot \vec{k} = [L[f_\sfx,
\Gamma_{\sft}]]\cdot (\vec{g}-\vec{e})= [L[f_\sfx, \Gamma_{\sft}]]\cdot
\vec{g} + \vec{d},\]
where $d\in D:=-L[f_x,\zG_{\sft}]\cdot E$, a finite set of vectors in
$\R[f_\sfx^{-1}(\zG_{\sft})]$ with negative coordinates.

We have proved the following proposition.

\begin{prop}\label{prop:gtilde} Let $(f_x,\zG_{\sft})\in \sfX\times
\sfT$.  Let $E$ and $H$ be the finite subsets of $\op \cG$ defined
immediately after Proposition~\ref{prop:AtL}.  Let $C:=H \cup
(\sfX(E^{-1})\cdot H^{-1})$, and let
$D:=-L[f_x,\zG_{\sft}]\cdot E$.  Diagram
\ref{eqn:gtilde} determines a map $g\mapsto \widetilde{g}$ from
$\text{Tw}(\zG_{\sft})$ to $\text{Tw}(\zG_{\widetilde{t}})$ such that
\begin{enumerate}
  \item $\widetilde{g}= c_1\cdot (f_\sfx @ g) \cdot c_2$ with 
$c_1, c_2 \in C$ and 
  \item $\vec{\widetilde{g}}= L[f_\sfx, \Gamma_{\sft}]\cdot
\vec{g} + \vec{d}$ with $\vec{d}\in \vec{D}$.
\end{enumerate}
\end{prop}
\smallskip

\subsection{Spectra} \label{subsecn:spectra} Let $a \in L[B]$.  Then
$a$ is an element of $\text{End}(\cV)$ and so we may consider its
spectrum.  For a vector space with a basis, an element $v$ is called
\emph{non-negative} if its coefficients are non-negative; we write
this as $v \geq 0$.  The linear map $a$ is a non-negative linear map,
in the sense that $v \geq 0 \implies a(v) \geq 0$. Equivalently, the
matrix defining $a$ has non-negative entries.  The Perron-Frobenius
theorem implies that each $w \in L[B^*]$ has a nonnegative real
eigenvalue equal to its spectral radius, which we denote by
$\sigma(w)$.

\begin{prop} \label{prop:spectra} Suppose $a_1, \ldots, a_n$ is a
sequence of nonzero elements in $L[B]$ and $w:=a_n\circ \cdots \circ
a_1\in L[B^n]$. Then $\sigma(w)\neq 0$ if and only if $w$ is induced
by a closed edge-path in $\cS[B]$.
\end{prop}

\begin{proof} If $w$ is not induced by an edge-path, then $w \equiv 0$
in $\text{End}(\cV)$ by Proposition \ref{prop:whenzero}. If it
represents an edge-path which is not closed, then $w^2=0$ since its
support and image intersect only at the origin.  Conversely, if $w$ is
induced by a closed edge-path of length $n$, say starting and ending
at $\cG.\Gamma$, then the corresponding linear map $a$ is represented by
$L[F, \Gamma]$ for some $F \in B^n$ and multicurve $\zG$ for which
$F^{-1}(\Gamma)\in {\op \cG}.\Gamma$, by Proposition
\ref{prop:whenzero}.  The nonnegative linear map $a$ cannot be
nilpotent: if it were, then after conjugation by a permutation matrix,
it would be upper-triangular, hence have a zero row;
cf. \cite[Thm. 0]{MR1212727}.  This is impossible by the remark a bit
before Proposition~\ref{prop:choices}.

\end{proof}

We equip $\cV$ with the Euclidean norm, and use this to define the operator norm $||\cdot ||$ on elements of $\text{End}(\cV)$. 
\begin{defn}
\label{defn:jsr}
For $n \geq 1$, set 
\[ \whatsigma_n(L[B]):=\max\{||w|| : w \in L[B^n]\}.\]
The \emph{joint spectral radius} is 
\[ \whatsigma(L[B]):=\lim_{n\to\infty}\whatsigma_n(L[B])^{1/n}.\]\
\end{defn}
The limit is independent of the chosen norm, since we are dealing with finite-dimensional spaces.

In particular, $\whatsigma(L[B])<1$ if and only if there are constants
$C>0$ and $0<\lambda<1$ such that $||w||<C\lambda^n$ for all $n \geq
1$ and $w\in L[B^n]$. In turn, this occurs if and only if for any
norm, there is a positive integer $q$ such that for any $n \geq q$ and
any $w \in L[B^n]$, we have $||w||< 1/2$.

The limit always exists, and is independent of the chosen norm.  For more on the joint spectral radius, see \cite{Cicone:2015aa} and the references therein.
We remark that even for matrices with binary or equivalently nonnegative rational entries, the problem of deciding whether the joint spectral radius is at most 1 is algorithmically undecideable (\emph{op. cit.}, p. 21, p. 30). And the question of realizibility of the joint spectral radius by some finite product (with the corresponding power)--called having the finiteness property--is also difficult; see \cite{MR2003315}. If such realizability is known to hold, the question of when the joint spectral radius is less than 1 becomes algorithmically decideable (\cite{Cicone:2015aa} p. 27). 

The following is essentially \cite[Thm. I(b), p. 22]{MR1152485}.
\begin{lem}
\label{lem:tozero}
The joint spectral radius satisfies $\whatsigma(L[B])<1$ if and only
if for any infinite sequence $a_1, a_2, a_3, \ldots$ we have
$||a_n\circ \cdots\circ a_2\circ a_1|| \to 0$ as $n \to \infty$.  
\end{lem}

\subsection{Rationality}

Thurston's characterization of rational functions \cite{DH1} and the preceding functoriality imply the following.

\begin{prop}
\label{prop:ratB}
Suppose $B$ is a non-exceptional Hurwitz biset.
\begin{enumerate}
\item Every element of $B$ is rational if and only if for each $a \in L[B]$ arising as the weight of a cycle of length one, the spectral radius satisfies $\sigma(a)<1$.
\item Given $n$, every element of the biset $B^n$ is rational if and only if  for each $w=a_n\circ\cdots\circ a_1$ arising as the weight of a cycle of length $n$, the spectral radius satisfies $\sigma(w)<1$.  
\item Every element of the semigroup $B^*$ is rational if and only if for every $n$ and each $w=a_n\circ\cdots\circ a_1$ arising as the weight of a cycle of length $n$, the spectral radius satisfies $\sigma(w)<1$.  
\item If $\whatsigma(L[B])<1$, then every element of $B^*$ is rational. 
\end{enumerate}
\end{prop}

\section{Contraction on $\cG$ implies contraction on $\cV$}
\label{secn:twistgroups}

In this section, we prove that if $B$ is contracting on ${\op \cG}$, then $\whatsigma(L[B])<1$. 

\subsection{Algebraic preliminaries}
\label{subsecn:prelim}

We return to a general $G$-biset $B$ that is contracting on a general
group $G$.  Let $\sfX$ be a basis of $B$, and fix a word norm
$|\cdot|$ on $G$. The \emph{contracting coefficient} is
\[\widehat{\rho}(B):=\limsup_{n\to \infty} \sqrt[n]{ \limsup_{|g|\to
\infty}\max_{v \in \sfX^n} \frac{|v@g|}{|g|}}.\] 
The contraction coefficient is independent of the norm on $G$ and of
the chosen basis, and $B$ is contracting on $G$ if and only if
$\widehat{\rho}(B)<1$; see \cite[Prop. 4.3.12]{MR4475129}.

Unwinding the definition, we see that for any $\rho$ with
$\widehat{\rho}<\rho<1$, there exist $n_0$ and $R_0$ such that for all
$g \in G$ with $|g| > R_0$, and all $v\in \sfX^*$ with $|v| \geq n_0$,
\[ |v@g| < \rho^{|v|}|g|. \]

Let $\rho, n_0, R_0$ be as above. 

\begin{prop}
\label{prop:basiccont}
Suppose $B$ is a contracting left-free-covering $G$-biset with basis $\sfX$.
\begin{enumerate}
\item For any finite set $C \subset G$, there exists a finite set $\widehat{C} \supset C$ such that $\sfX\widehat{C} \subset \widehat{C}\sfX$. Equivalently, $\sfX(\widehat{C}) \subset \widehat{C}$. 
\item For a finite set $C$, the iterated function system on $G$ given by the collection of maps 
\[ \calL:=\{ \;g \mapsto c_1 (\sfx @ g) c_2 \;| \;c_1, c_2 \in C, \sfx \in \sfX\}\]
is uniformly contracting: for any $n$ and $\ell_1, \ldots, \ell_n \in \calL$, we have 
\[ |\ell_1\ell_2\cdots \ell_n (g)| \le\rho^{n}|g|+O(1)\]
where the implicit constant is independent of $g$ and $n$. In particular, it 
has a nucleus: a finite set $\cN(\calL)$ depending on $C$ and on $\sfX$ such that for all $g \in G$, there exists $N \in \N$ such that for any $n \geq N$ and any sequence $\ell_1, \ell_2, \ldots, \ell_n \in \calL$, we have 
\[ \ell_1\ell_2\cdots \ell_n (g) \in \cN(\calL).\]
\end{enumerate}
\end{prop}

\begin{proof} (1) Let $B(R)$ be the ball in $G$ with radius $R$.  We
choose $R$ so that $R\ge R_0$, $C\subset B(R)$ and
$\sfX^{n_0}(B(R_0))\subset B(R)$.  The contraction inequality
$|v@g|<\rho^{|v|} |g|$ implies that
$\sfX^{n_0}(B(R)-B(R_0))\subset B(R)$.  Thus $\sfX^{n_0}(B(R))\subset
B(R)$.  Setting $\widehat{C}:=B(R) \cup \sfX(B(R)) \cup \cdots \cup
\sfX^{n_0}(B(R))$, it follows that $\sfX(\widehat{C}) \subset
\widehat{C}$.

(2) By (1) we may assume by enlarging $C$ that $\sfX(C) \subset C$ and
$\cN(\sfX)\subset C$, where $\cN(\sfX)$ is the nucleus for the
$G$-biset $B$ with respect to the basis $\sfX$.  An easy induction
argument shows that
 \[ \ell_1\ell_2\cdots \ell_n(g) \in C \sfX(C)
\cdots \sfX^n(C) \sfX^n(g) \sfX^n(C) \cdots \sfX(C)C.\]
Since $C$ is finite and $B$ is contracting, there exists $n_1$ such
that $n\geq n_1$ implies $\sfX^n(C) \subset \cN(\sfX)$.  Since
$\cN(\sfX)\cN(\sfX) \subset \cN(\sfX)$ and more generally $\cN(\sfX)^k
\subset \cN(\sfX)$ for all $k \geq 2$, we conclude the factors to the
left and right of the set $\sfX^n(g)$ in the above expression
stabilize.  For $n \geq n_1$, we have \[ \ell_1\ell_2\cdots\ell_n(g)
\in C \sfX(C) \cdots \sfX^{n_1-1}(C)\cN(\sfX)\sfX^n(g)\cN(\sfX)X^{n_1-1}(C)\cdots X(C)C.\]
Since $\sfX(C) \subset C$ and $\cN(\sfX)\subset C$, we conclude \[
\ell_1\ell_2\cdots\ell_n(g) \in C^{n_1+1} \sfX^n(g) C^{n_1+1}.\] Put
$m:=\max\{|g| : g \in C\}$. If $g\notin C$, then $g\notin B(R_0)$ and
so $|v@g| < \rho^{|v|}|g|$ for every $v\in \sfX^n$.  Hence \[
|\ell_1\ell_2\cdots \ell_n(g)| \leq 2(n_1+1)m+(\rho^{n}|g|+m)\] and
the conclusion follows.

\end{proof}

Suppose now $B$ is a contracting Hurwitz biset.  Let $\sfX$ be a basis of $B$.  To ease notation in what follows, we write simply $\sfx$ for elements of $\sfX$, instead of $f_{\sfx}$.
The implication (1) implies (3) in the statement of Theorem \ref{thm:main} will follow from 

\begin{prop}
\label{prop:coeffs}
Suppose $B$ is a contracting Hurwitz biset.  Then $\widehat{\sigma}(L[B]) \leq \widehat{\rho}(B)$. 
\end{prop}

Our next result will be to relate two notions of size for multitwists.
Recall that if $\zG$ is a multicurve, then we have a group isomorphism
$\zi_\zG\co \text{Tw}(\zG)\to \Z[\zG]\subseteq \R[\zG]$ and that
$\R[\zG]$ may be identified with a subspace of $\cV$.  We equip
$\R[\Gamma]$ with the $\ell^1$-norm, so that $|\sum_{\zg\in \zG}a_\zg
\zg|:=\sum_{\gamma \in \Gamma}|a_\gamma|$.  We will need the following
result; see \cite[\S 2.3]{MR1813237} and \cite[Cor. 4]{MR2471596}.  We
suppose $\cG$ is equipped with some generating set; we denote by
$|g|_\cG$ the resulting norm of an element $g \in \cG$.

\begin{prop}
\label{prop:undistorted} For any multicurve $\Gamma$, the subgroup $\Tw(\Gamma)$ is undistorted in $\cG$.  That is, there is a constant $c$ depending on the generating sets for $\cG$ and $\Tw(\Gamma)$ such that for any $g \in \Tw(\Gamma)$, we have 
\[ |\zi_\zG(g)| \leq c |g|_\cG.\]
\end{prop}
\begin{proof}[Suketch of a proof] Using the lantern relation \cite[Proposition
5.1]{fm} and the description of $\cG$ in \cite[Section 9.3]{fm}, one can
prove that $\Tw(\Gamma)$ injects into the Abelianization $\cG^{ab}$ of $\cG$.  The norm on $\cG$ descends to a norm on $\cG^{ab}$ such that the natural quotient map $\cG \ni g \mapsto g^{ab} \in \cG^{ab}$ is $1$-Lipschitz. Since subgroups of free abelian groups are always undistorted, we have 
\[ |\iota_\Gamma(g)| \asymp |g^{ab}| \leq |g|.\]
\end{proof}

Before returning to ${\op \cG}$, we find a convenient generating set $S$ for $\cG$.  Let $\sfT$ be an orbit transversal for the action of $\cG$ on multicurves.  For each $\Gamma_{\sft} \in \sfT$, and for each $\gamma \in \Gamma_{\sft}$, we add the positive (right) Dehn twist $\twist(\gamma)$ about $\gamma$ to our set $S$, along with its inverse. 

We then extend this set of twists to a symmetric set of generators of $\cG$.  We denote by $|\cdot|_\cG$ the corresponding word norm on $\cG$.

By construction of the generating set $S$, for any $\Gamma \in \sfT$ and any $g \in \Tw(\gamma)$, we have $|g|_\cG \leq |\zi_\zG(g)|$.  Since our transversal $T$ consists of finitely many multicurves, we can take the maximum of the constants $c$ in Proposition \ref{prop:undistorted} over all elements of $T$, and conclude 

\begin{cor}
\label{cor:comparable}
For any $\Gamma \in \sfT$, and any $g \in \Tw(\Gamma)$, we have 
\[ |\zi_\zG(g)| \asymp |g|_\cG,\]
where the implicit constant is independent of $g$. 
\end{cor}
\medskip

\subsection{Proof of Proposition \ref{prop:coeffs}}
\label{subsecn:coeffs}

The strategy: by making finite corrections to each, we can identify
the IFS on $\cV$ given by $\cS[B]$ with the restriction of the IFS
on ${\op \cG}$ to twist subgroups about a representing set of multicurves.
After the correction, the IFS on $\cV$ will be affine, instead of
linear, and the IFS on ${\op \cG}$ will have elements that look like $g
\mapsto c_1\cdot (\sfx @g)\cdot c_2$, where $c_1, c_2$ are drawn from
a finite set. The philosophy is that such finite corrections do not
alter contraction ratios for either.

We now suppose we are in the setup of \S\S \ref{subsecn:ifs} and \ref{subsecn:prelim}; we fix $\rho$ with $\widehat{\rho}(B)<\rho<1$.  Suppose $a_1, a_2, \ldots$ is a sequence of elements of $L[B]$.  Fix $n \in \N$.  Our goal is to show that the operator norms of the compositions satisfy 
\begin{equation}
\label{eqn:control}
||a_n \circ\cdots \circ  a_2 \circ  a_1|| \le \rho^n\; \text{as} \; n \to\infty.
\end{equation}
From the definition of $L[B]$, for each $n \geq 1$, there exists
$\Gamma_n \in \sfT$ and $\sfx_n \in \sfX$ such that $a_n \in
\text{End}(\cV)$ is induced by some $L[\sfx_n, \Gamma_n]: \R[\Gamma_n]
\to \R[\Gamma_n']$ where $\Gamma_n':=\sfx_n^{-1}(\Gamma_n)$.  

Suppose now we are given $g \in \Tw(\Gamma_1)$, which we identify with
$\zi_\zG(g) \in \R[\Gamma_1]$. Put $g_1:=g$. We define now inductively
a sequence $g_n \in \Tw(\Gamma_n)$ by setting
$g_{n+1}=\widetilde{g}_n$, using the map in
Proposition~\ref{prop:gtilde}.

Proposition~\ref{prop:gtilde} implies that, on the one hand, as
elements of ${\op \cG}$, the terms of the sequence $(g_n)_n$ form an orbit
of the IFS on ${\op \cG}$ given by the collection of operators $\{g \mapsto
c_1\cdot \sfx@g \cdot c_2 : \sfx \in \sfX, c_1, c_2 \in C\}$ as in
Proposition \ref{prop:basiccont}.  Using the notation $\vec{g}_n$ as
in Proposition~\ref{prop:gtilde} and applying Corollary
\ref{cor:comparable}, we conclude
\begin{equation}
\label{eqn:gseqest}
\begin{array}{lcr}
|\vec{g}_n| \asymp |g_{n}|_\cG & \le & \rho^n|g|_\cG +O(1).
\end{array}
\end{equation}
On the other hand, as vectors, Proposition~\ref{prop:gtilde} implies
that we have 
\[\vec{g}_{n+1}= L[\sfx_n, \Gamma_n]\cdot
\vec{g}_n + \vec{d}_n=a_n\cdot \vec{g}_n+\vec{d}_n,\]
where the translation term $\vec{d}_n$ is drawn from a finite set,
$\vec{D}$.\footnote{The sets $C$ and $\vec{D}$ represent rather
different objects, but play similar roles in being finite
``corrections''.}  We write this succinctly as
\[ \vec{g}_{n+1}=M_n\cdot \vec{g}_n,\]
where now the map $M_n: \cV \to \cV$ is affine, rather than linear. 

In summary, the terms of the sequence $(\vec{g}_n)_n$ form an orbit of
the IFS on $\cV$ given by the collection of operators $\{M(\vec{v}) :=
a(\vec{v})+\vec{d} : a \in L[B], \vec{d} \in \vec{D}\}$.

We have for all compositions of a fixed length $n$ and every $v\in\cV$ that
\[ a_n \circ\cdots \circ  a_2 \circ  a_1 (\vec{v})=M_n \circ\cdots \circ  M_2 \circ  M_1(\vec{v})+O(1)\]
so
\[ |a_n \circ\cdots \circ  a_2 \circ  a_1(\vec{v})| \le |M_n \circ\cdots \circ  M_2 \circ  M_1(\vec{v})|+O(1).\]
Taking $\vec{v}=\vec{g}_1$, this and the estimate (\ref{eqn:gseqest}) 
imply that 
\[ |a_n \circ\cdots \circ  a_2 \circ  a_1(\vec{g}_1)| \le
|\vec{g}_n|+O(1) \le b\rho^n|g_1|_\cG+O(1)\le
c\zr^n|\vec{g}_1|+O(1)\]
for multiplicative constants $b$ and $c$ which are universal in
$n$ and an implied constant which depends on $n$.  It follows that
the lim sup as $|\vec{g}_1|\to \infty $ of the quotients 
\[|a_n \circ\cdots \circ  a_2 \circ  a_1(\vec{g}_1)|/|\vec{g}_1|\] 
is at most $c \zr^n$. 
We conclude that $\widehat{\sigma}(L[B]) \leq \rho$. Letting $\rho
\downarrow \widehat{\rho}(B)$, we conclude that the joint spectral
radius $\widehat{\sigma}(L[B])$ is bounded by the contraction factor
of the Hurwitz biset $B$. This completes the proof.

\section{Correspondence on moduli space and IFS on $\cT$}
\label{sec:corr}

Central to our development is the fact that Teichm\"uller spaces of marked spheres admit natural descriptions both in terms of maps and in terms of paths.  

\subsection{Moduli and Teichm\"uller spaces} 
\label{subsec:modteich}
Fix a subset $P \subset S^2$, this time with $\#P\geq 4$, and now identify $S^2$ with the complex projective line $\Pone$.  Here, the restriction on the cardinality is to eliminate mention of uninteresting special cases.  As a set, the moduli space $\cM:=\Moduli(S^2, P)$ is defined to be the set of injections $\iota: P \hookrightarrow \Pone$ modulo the action of $\Aut(\Pone)$ by post-composition; it comes with a basepoint $\star$ represented by the inclusion $P \hookrightarrow \Pone$.  As a complex space, it is isomorphic to a hyperplane complement in $\C^{\#P-3}$.  

The Teichm\"uller space $\cT:=\Teich(S^2, P)$ admits two distinct descriptions which, due to our choice of identification $S^2=\Pone$, are canonically identified.  As this is well-known, our treatment here is brief; see \cite[\S 2]{MR3054977} for details, which references the key fact of the contractability of homeomorphism groups of spheres fixing three marked points.  

On the one hand, we have a description as a double coset space
\[ \cT^{\text{maps}}:=\Aut(\Pone) \backslash \Homeo(S^2)/\Homeo_0(S^2,P)=
\{\text{homeos.}\; \tau: S^2 \to \Pone\}/\sim\]
where $\tau_1 \sim \tau_2$ if there exists $M \in \Aut(\Pone)$ with $\tau_2$ isotopic to $M \circ \tau_1$ through homeomorphisms agreeing on $P$.  

Recalling that the universal cover of a (suitably nice) space is constructed as homotopy classes of paths from a basepoint, we have a second description 
\[ \cT^{\text{paths}}:=\{\tau:[0,1] \to \cM,  \; \tau(0)=\star\}/\simeq\]
where now $\simeq$ denotes homotopy relative to endpoints. 

There is a natural map $\cT^{\text{maps}}\to \cT^{\text{paths}}$ given as follows.  Given $\tau \in \cT^{\text{maps}}$ there is a unique Teichm\"uller geodesic joining the basepoint $*=\text{id}$ to $\tau$. This gives a one-parameter family of homeomorphisms which, when evaluated at $P$, gives a motion of $P$ and hence a path $\tau$ in moduli space starting at the basepoint $*=\text{id}_P$.   Alternatively: forgetting all but three points of $P$, the resulting space of homeomorphisms is contractible, so there is an isotopy from $\tau$ to the identity relative to this set of three points; evaluating this isotopy on $P$ gives a path to the basepoint whose reverse is the desired element of $\cT^{\text{paths}}$.  There is a similarly natural map $\cT^{\text{paths}}\to \cT^{\text{maps}}$ given by isotopy extension of a motion of $P$.  Abusing notation, we use the same symbol $\tau$ to stand for an element of $\cT$, a representing homeomorphism $S^2 \to \Pone$, and a representing path in moduli space starting at $\star$. 

We equip both $\cT$ and $\cM$ with the Teichm\"uller length metrics; they are both complete and unbounded.  We denote balls in these metrics by $B(\; , \;)$ and the length of a rectifiable path by $|\; |$. 

\subsection{Mapping class group as fundamental group of moduli space}  
\label{subsec:mcg}

The fundamental group of moduli space also admits twin descriptions in terms of maps and of paths.  On the one hand,
\[ \cG:=\pi_1(\cM, \star)^{\text{maps}}:=\{g: (S^2, P) \to (S^2, P)\}/\sim\]
where the equivalence is isotopy relative to $P$. On the other,
\[ \pi_1(\cM, \star)^{\text{paths}}:=\{g: [0,1] \to \cM, g(0)=g(1)=\star\}/\sim\]
where the equivalence is homotopy relative to the endpoints at $\star$.  Again there are canonical identifications between these descriptions, but with a caveat.  As maps, the expression $g_1\circ g_2$ means apply $g_2$ first.  However, as paths, the expression $g_1\cdot g_2$ means apply $g_1$ first.  So when group operations are taken
into account, $\pi_1(\cM,*)^{\text{paths}}={\op \cG}$.  We again abuse notation so that
e.g. the symbol $g$ stands simultaneously for an element of ${\op \cG}$, a
representing map, and a representing loop based at $\star$.

\subsection{The  action of $\cG$ on $\cT$}  
\label{subsec:mcgact}
Suppose $g \in \cG$ and $\tau \in \cT$. We wish the induced map $\sigma_g: \cT \to \cT$ to be given by lifting complex structures under $g$. So when $g$ and $\tau$ are represented by maps, we set  
\[ g.\tau:=\tau \circ g.\]
This defines a left action of ${\op \cG}$ since
\[ (g_1 \cdot g_2).\tau = (g_2 \circ g_1). \tau:=\tau \circ (g_2 \circ g_1) = (\tau \circ g_2) \circ g_1 = g_1.(g_2.\tau).\]
When represented by paths, we set 
\[ g.\tau:=g \cdot \tau.\]
This also defines a left action of ${\op \cG}$ on $\cT$ since 
\[ (g_1 \cdot g_2).\tau :=(g_1 \cdot g_2)\cdot \tau=g_1\cdot(g_2 \cdot \tau)=g_1.(g_2.\tau).\]
The two interpretations coincide.  With the operation $\cdot$ the one in ${\op \cG}$, we have in either interpretation
\[ \sigma_{g_1\cdot g_2} = \sigma_{g_1} \circ \sigma_{g_2}.\]

In \S \ref{subsecn:two} below, we record similar results for Hurwitz bisets. 

\subsection{The action of $B$ on $\cT$}
\label{subsec:liftB}
 Again, below we omit, for example, details showing that various maps are well-defined, etc. See \cite{DH1} and \cite{koch:endos}.  Suppose $B$ is a Hurwitz biset. 
An element $f \in B$ determines a holomorphic self-map $\sigma_f: \cT \to \cT$  by pulling back complex structures:  
\[
\xymatrix{
(S^2,P) \ar@{.>}[r]^{\tilde{\tau}=\sigma_f(\tau)} \ar[d]_{f}& (\Pone, P_{\tilde{\tau}}) \ar@{.>}[d]^{R}\\
(S^2,P) \ar[r]_(.48){\tau} & (\Pone, P_\tau).
}
\]
In the above diagram, the dotted arrows are induced, and $R$ is a rational map.    

Stacking two such diagrams on top of each other, we see that for $f_1, f_2 \in B$ we have 
\[ \sigma_{f_2 \circ f_1}=\sigma_{f_1} \circ \sigma_{f_2}.\]
Thus the semigroup ${\op B}^{\otimes *}$ acts naturally on the left on $\cT$:
\[ (f_1 \otimes \cdots \otimes f_n).\tau:=\sigma_{f_1}\circ \cdots \circ \sigma_{f_n}(\tau).\] 

\subsection{Correspondence on moduli space} \label{subsec:corr} We
collect facts from \cite{koch:endos} and \cite{kmp:kps}.  There is a
finite-index subgroup $\cG_f<\cG$--the ``liftables'' for $f$--defined
on representing maps so that 
\[ \cG_f:=\{h\in \cG : \exists \tilde{h}\in \cG, \;
h\circ f \simeq f\circ \tilde{h}\}.\] 
The diagram 
\[\xymatrix{ (S^2, P) \ar[r]^{\tilde{h}} \ar[d]^{f} & (S^2,P)
\ar[r]^{\tilde{\tau}=\sigma_f(\tau)} \ar[d]_{f}& (\Pone,
P_{\tilde{\tau}}) \ar[d]^{R}\\ (S^2, P) \ar[r]_{h}& (S^2,P)
\ar[r]_(.48){\tau} & (\Pone, P_\tau) } \] 
makes it clear that the map $\sigma_f$ covers a correspondence on
moduli space

\[\xymatrix{ &
\cT \ar[dd]_{\pi}\ar[rr]^{\sigma_f} \ar[dr]^{\zeta} &
&\cT \ar[dd]^{\pi} \\ &&\mathcal{W}\ar[dl]_{\phi}\ar[dr]^{\rho}
&\\ & \cM& & \cM.}\] 
In the diagram above, $\pi$ is the universal covering projection, $\zeta$ is the covering projection to $\cW:=\cT/\cG_f$, and $\phi$ is a finite cover.  The map $\rho$ is holomorphic; it may be constant, injective, a Galois surjective cover onto the whole of moduli space, or something in between; see \cite{bekp}.  The definitions imply that this correspondence $\phi, \rho: \cW \to \cM$ depends only on the Hurwitz biset $B$ of $f$, as defined in \S \ref{secn:hb}.
The covering $\phi$ is determined, up to equivalence, by the conjugacy class in $\cG$ of the finite-index subgroup $\cG_f<\cG$
so that $[\cG:\cG_f]=\deg(\phi)$. 

The map $f$ is conjugate up to isotopy to a rational function if and only if $\sigma_f$ has a fixed-point in $\cT$. 
The property of $f$ being non-exceptional implies that if $p=\#P$, the iterate $\sigma_f^{\circ p}$ is a strict  (though typically not uniform) contraction, and so a fixed-point, if it exists, must be unique; see \cite{MR3289714} and \S \ref{subsecn:contraction} below.

We denote by $\psi:=\phi \circ \rho^{-1}: \cM \dashrightarrow \cM$ the
corresponding multi-valued map.  The map $\psi$ depends only on $B$,
not on the choice of $f\in B$.  Since $\phi$ is a covering, $\psi$ has
the path-lifting property; this is the reason for our choice of direction for $\psi$.

\subsection{Two canonical models of Hurwitz bisets} 
\label{subsecn:two}

Fix a Hurwitz $\cG$-biset $B$; we denote by ${\op B}$ the corresponding ${\op \cG}$-biset.   Let $\phi, \rho: \cW \to \cM$ be the correspondence of \S \ref{subsec:corr}.  Let $\star \in \cM$ be the basepoint, and fix an identification $\phi^{-1}(\star) \leftrightarrow \sfX$.  Having now in hand the induced correspondence on moduli space, we now recall how $B$ admits canonical interpretations in terms of either maps or paths.

Recall that from the definition, on the one hand, 
\[ B=B^{\text{maps}}:=\{g_0 \circ f \circ g_1 : g_0, g_1 \in \cG\}/\text{isotopy relative to $P$}\]
with the left and right $\cG$-actions by pre- and post-composition with representing maps, respectively.

On the other hand, applying the construction in \cite[\S 4]{MR3781419}, we obtain a biset 
\[ B^{\text{paths}}:=\{(\tau, \sfx) : \tau:[0,1]\to \cM, \tau(0)=\star, \rho(\sfx)=\tau(1), \phi(x)=\star\}/\approx\] 
where $(\tau, \sfx) \approx (\tau', \sfx')$ if and only if $\sfx=\sfx'$ and $\tau,
\tau'$ are homotopic relative to their endpoints. The left action of ${\op \cG}$ by
pre-concatenation of paths $g\cdot (\tau,\sfx)=(g\cdot \tau,\sfx)$ is free.  The right
action is defined as $(\tau,\sfx)\cdot g = (\tau \cdot \rho(\widetilde{g}[x]),y)$, where $\widetilde{g}[\sfx]$ is the lift of $g$ under $\phi$ based at $\sfx$ joining $\sfx$ to some $\sfy \in \sfX$. 

The following fact--a generalization of the canonical identification
of ${\op \cG}$ with $\pi_1(\cM, *)^{\text{paths}}$--is essential to our development.  A more general version of the first statement, which deals with maps $f: (S^2, C) \to (S^2, A)$ for finite sets $C, A \subset S^2$, has appeared as \cite[Thm. 9.1]{MR4213769}. The second and third assertions follow from functoriality.

\begin{prop}
\label{prop:identify} 
The following admit canonical identifications:
\begin{enumerate}
\item the ${\op \cG}$-bisets ${\op B}^{\text{maps}}$ and $B^{\text{paths}}$;
\item for each $n \in \N$, the ${\op \cG}$-bisets $({\op B}^{\text{maps}})^{\otimes
 n}$ and $(B^{\text{paths}})^{\otimes n}$;
\item the semigroups $({\op B}^{\text{maps}})^{\otimes *}$ and
$(B^{\text{paths}})^{\otimes *}$. 
\end{enumerate}
\end{prop}

In general, the assignment of a biset of paths to a covering self-correspondence is sufficiently functorial so that the biset of an iterate is given by suitable concatenations of paths.  Let us now see how this works in our setting. 

Suppose $f_1, f_2 \in B^{\text{paths}}$ are represented by $(\tau_1, \sfx_1)$
and $(\tau_2, \sfx_2)$, respectively.   There is a unique lift
$\widetilde{\tau_2}[\sfx]$ of $\tau_2$ under $\phi$ based at $\sfx_1$. Under
$\rho$ this projects to a path starting at the endpoint
$\tau_1(1)=\rho(\sfx_1)$.  The concatenation $\tau_1 \cdot
\rho(\widetilde{\tau_2}[\sfx_1])$ defines an element of
$(B^{\text{paths}})^{\otimes 2}$. This is well-defined and respects the ${\op \cG}$-actions. 
 Generalizing our previous conventions, we denote this path simply by $f_1\cdot f_2$.  This extends more generally to arbitrary products of finitely many elements. 

We remark here that the concatenation of any pair of elements drawn from 
$(B^{\text{paths}})^{\otimes *} \cup {\op \cG}$ in fact admits the unifed description
\[ a \cdot b := a \cdot \tilde{b}[a]\]
where $\tilde{b}[a]$ denotes the unique lift of $b$ under $(\phi \circ
\rho^{-1})^{\circ |a|}$ starting at the terminus of $a$, where $|a|=0$ if $a
\in {\op \cG}$ and $|a| = n$ if $a \in B^{\otimes n}$, and the superscript $\circ |a|$ denotes iteration.

Furthermore, the action of $(B^{\text{paths}})^{\otimes *} \cup {\op \cG}$ on $\cT$ is naturally a left action: 
\[ (a\cdot b) . \tau = a.(b.\tau) = \sigma_a\circ \sigma_b(\tau)=\sigma_{a\cdot b}(\tau).\]
So in particular for arbitrary elements $a, b, c, \ldots, z \in (B^{\text{paths}})^{\otimes *} \cup {\op \cG}$ we have, evaluating at the basepoint $\star$ of $\cT$, that 
\[ \sigma_{a\cdot b \cdot c \cdot \cdots \cdot z}(\star) = \sigma_a \circ \sigma_b \circ \sigma_c \circ \cdots \circ \sigma_z(\star)\]
and that the projection of this image to $\cM$ coincides with the endpoint of the path $a\cdot b \cdot c \cdot \cdots \cdot z$.

\subsection{Contraction properties}\label{subsecn:contraction}See \cite{DH1} and \cite[\S 2.5]{MR3289714}. For any Hurwitz biset $B$, the maps $\rho$ and, equivalently, $\sigma_f$ for some (any) $f \in B$, do not increase Teichm\"uller distances, since they are holomorphic.  When $P$ is the postcritical set of $f$, taking the second iterate suffices to obtain a strict  non-uniform contraction.  For general $P$, we  may need to take $p:=\#P$ iterates  to obtain strict  contraction of $\sigma_f^{\circ p}$.  Formally, our setup is a little different: we are considering compositions, not iterates. Nevertheless, we still have:

\begin{prop}
\label{prop:strict}
Suppose $B$ is a non-exceptional Hurwitz biset, and $p:=\#P$. 
\begin{enumerate}
\item For any $F:=f_1\circ \cdots \circ f_p \in B^{\circ p}$, the map $\sigma_{F}$ is a strict  contraction with respect to the Teichm\"uller metric. 
\item Given any compact nonempty subset $K \subset \cM$, there exists
$\lambda:=\lambda(K)\in [0,1)$ such that for any $F \in B^{\circ p}$, we have 
\[ \sup\{ | d\sigma_F(\tau)| : \tau\in \cT,\pi(\tau) \in K\} < \lambda\]
where $| \cdot |$ is the operator norm with respect to the Teichm\"uller metric. 
\end{enumerate}
\end{prop}

\begin{proof} To see (1), note that since we are dealing with pure
Hurwitz bisets, the $f_i$'s all agree on $P$. By varying within
isotopy classes we may therefore assume $P \cup \text{Crit}(F)=P \cup
\text{Crit}(f_1^{\circ p})$. By \cite[Lemma 2.8]{MR3289714}, which references only the set-theoretic dynamics on these common subsets, since $f_1$ is non-exceptional, if $E \subset P$ is any set with $\#E\geq 4$, then there exists $e \in E$ with $f_1^{-p}(e) \not\subseteq \text{Crit}(f_1^{\circ p})\cup P$.  From this the proof proceeds in the identical fashion as in the case of a single map. 

For (2), suppose $\sfX$ is a basis for $B$.  Then given any $F \in
B^{\circ p}$ there exist unique $g \in {\op \cG}$ and $w \in \sfX^p$
with $F=g\cdot w$.  Thus $\sigma_F=\sigma_g \circ \sigma_w$.  Since
$\sigma_g$ is an isometry, we conclude that $|d\sigma_F| =
|d\sigma_w|$.  The claim now follows from (1) and the observation that
$\sfX^p$ is finite.

\end{proof}

\subsection{Limit sets}\label{subsecn:limitsets} The correspondence $\phi, \rho: \cW \to \cM$
associated to a Hurwitz biset $B$ can be iterated, and so defines a
dynamical system on moduli space.   An orbit  corresponds to a
sequence of iterated inverse images of the multivalued map
$\psi=\phi\circ \rho^{-1}$.   More formally: a backward orbit is a sequence $(w_0, w_1, \ldots)$ of elements of $\cW$ such that for each $n=1, 2, \ldots$, we have $\mu_n:=\rho(w_{n-1})=\phi(w_{n})$.  Abusing notation, and setting $\mu_0:=\phi(w_0)$, we denote a backward orbit by the sequence $(\mu_0, \mu_1, \ldots)$ with the implicit understanding that this arises from such a sequence $(w_0, w_1, \ldots)$.  
Via functorial pullbacks, the $n$th iterate of the correspondence is given by a pair $\phi_n, \rho_n: \cW_n \to \cM$, where now $\phi_n$ is a cover of degree $(\deg\phi)^n$, and $\rho_n$ is holomorphic.  So equivalently, with analogous conventions of notation, a backward orbit segment of length $n$ is a choice of preimage of $\mu_0$ under the multivalued map $\phi_n \circ \rho_n^{-1}$.  This latter perspective is useful, since it implies the following fact, which follows directly from the fact that $\phi_n$ is a cover and $\rho_n$ is distance-non-increasing. 

\begin{lem} \label{lem:inj} Suppose $(\mu_0, \mu_1, \ldots)$ is a
backward orbit, and $r>0$.  Suppose $B(\mu_0, 2r)$ is
simply-connected. Set $\overline{B}_0:=\overline{B(\mu_0, r)}$.  For
each $n$, let $\overline{B}_n:=\psi^{-n}(\overline{B}_0)$ denote the
lift of $\overline{B}_0$ based at $\mu_n$. Then $\psi^{-n}:
\overline{B}_0 \to \overline{B_n}$ is single-valued, and
$\diam(\overline{B}_n) \leq \diam(\overline{B}_0)$ for each $n$.
\end{lem}

We are also concerned with the full backward orbit $\cO^{-}(\mu_0):=\{(\mu_0, \mu_1, \ldots) : \psi(\mu_n)=\mu_{n-1}, n=1, 2, \ldots \}$ of a point $\mu_0 \in \cM$.  This forms a ``preimage tree'' with root $\mu_0$ and a directed edge joining a point to its image under $\psi$. This ``preimage tree'' is distinct from the ``coding tree'' that we will define below. 

\begin{defn}
\label{defn:limitsetmod}
The \emph{limit set} of the correspondence on moduli space $\phi, \rho: \cW \to \cM$ is the set $\cJ$ of accumulation points  in $\cM$ of $\cO^{-}(\mu_0)$, where $\mu_0:=\star$ is the basepoint.  More formally: $\mu \in \cJ$ if and only if for every neighborhood $U$ of $\mu$, there is a backward orbit $(\mu_0, \mu_1, \ldots)$ and a subsequence $n_1, n_2, \ldots$ such that $\mu_{n_k} \in U$ for infinitely many values of $k$. 
\end{defn}

\begin{prop}
\label{prop:limitset}
Suppose $\cJ \neq \emptyset$.
\begin{enumerate}
\item $\cJ$ is closed and is independent of the choice of seed: for any $\mu_0, \nu_0 \in \cM$, the accumulation points of $\cO^{-}(\mu_0)$ and of $\cO^{-}(\nu_0)$ in $\cM$ coincide.
\item $\cJ$ is backward-invariant: $\psi^{-1}(\cJ)\subset \cJ$.  Equality need not hold.
\item For each $\mu_0 \in \cJ$, the set $\cO^{-}(\mu_0)$ is dense in $\cJ$.
\item The set $\cJ$ is contained in the closure of the set of periodic points (which are necessarily repelling).
\item If in addition $\#\cJ \geq 2$, then $\cJ$ contains no isolated points. 
\end{enumerate}
\end{prop}

\begin{proof}
(1) That the limit set is closed follows easily from an elementary diagonalization argument. For independence of seed, suppose $\mu_{n_k} \to \mu \in \cJ$. Join $\mu_0$ to $\nu_0$ by a path $\gamma_0$ of length $L$.  Let $K:=\overline{B(\mu, 2L)}$.  For each $n$, let $\gamma_n$ be the lift of $\gamma_0$ based at $\mu_n$, and set $\nu_n$ to be the endpoint of $\gamma_n$. Then $\gamma_{n_k} \subset K$ for all sufficiently large $k$. By uniform contraction on $K$, we conclude $|\gamma_{n_k}| \to 0$ as $k \to \infty$. Thus $\nu_{n_k} \to \mu$. 

(2) Backward invariance holds because $\phi$ is a cover and $\rho$ is
continuous.  To see that the limit set $\cJ$ need not be equal to its
preimage under $\psi$, consider the following example.  The
correspondence induced by $f(z)=z^2+i$, in suitable coordinates, is
given by $\phi(w)=(2-w)^2/w^2$, $\rho(w)=w$,
$\cM=\Pone-\{0,1,\infty\}$; see
\cite{bartholdi:nekrashevych:twisted}. The point $\mu=2$ is in the
image of $\psi$, but it is not in the limit space $\cJ=\cM - \{2\}$.

(3) This follows from (1) and (2).  

(4) Suppose $\mu_0 \in \cJ$; by (3) there is a backward orbit $(\mu_0, \mu_1, \ldots)$  such that along some subsequence we have $\mu_{n_k} \to \mu_0=:\mu_\infty$.  Choose any $r>0$ so small that $K:=\overline{B(\mu_\infty, 3r)}$ is simply-connected. Set $B_0:=\overline{B(\mu_0, r)}$. By Lemma \ref{lem:inj} and uniform contraction over $K$, we have $\diam(B_n) \to 0$ as $n \to \infty$.  For some sufficiently large $n$ we then have $B_n \subset B_0$. So  $\psi^{-n}|_{B_0}: B_0 \to B_n \subset B_0$ is well-defined and a uniform contraction, so it has a fixed-point. 

(5) This follows (2) and the argument in (1), applied to two distinct
elements $\mu_0, \nu_0 \in \cJ$. 
 
\end{proof}

As we have defined $\cJ$, it is not obvious that $\cJ$ is nonempty, since
we have defined it in terms of accumulation points in a non-compact
space.

\subsection{IFS on $\cT$ and coding tree}\label{subsecn:codingtree} Suppose $B$ is a Hurwitz $\cG$-biset, $\phi, \rho: \cW \to \cM$ its correspondence over moduli space, and fix an identification $\phi^{-1}(\star)\leftrightarrow \sfX$. For each $\sfx \in \sfX$, abusing notation, fix a rectifiable path $\sfx: [0,1] \to \cM$ with $\sfx(0)=\star$ and $\sfx(1) = \rho(\sfx)$.  We now have a basis $\sfX$ of the ${\op \cG}$-biset $B^{\text{paths}}$ which, by Proposition \ref{prop:identify}, is identified with the ${\op \cG}$-biset ${\op B}^{\text{maps}}$.  We also get an IFS on $\cT$ generated by $\{\sigma_{\sfx}: \sfx \in \sfX\}$.  

Below, we denote the elements of $\sfX$ by $\sfx, \sfy, \ldots, \sfz$, etc. 

Given a word $w_n=\sfx_1\cdots\sfx_n \in \sfX^n$ we denote the image of the basepoint $\star$ in Teichm\"uller space under the action of $w_n$ by 
\[ \Lambda(w_n)=\Lambda(\sfx_1 \sfx_2 \cdots
\sfx_n):=\sigma_{\sfx_1}\circ \sigma_{\sfx_2}\circ \cdots \circ
\sigma_{\sfx_n}(\star)=:\sfx_1\cdot \sfx_2\cdot \cdots \cdot \sfx_n.\!\star.\]
We denote by $\sfX^{+\infty}$ the space of right-infinite strings
$\sfx_1 \sfx_2  \cdots $, equipped with the
product topology.  
The \emph{coding map} $\Lambda\colon\sfX^{+\infty }\to \cT$ is defined as the function which
takes an infinite word to the set of accumulation points of the
sequence of images of finite initial subwords of the word:
\[ \Lambda(\sfx_1\sfx_2\cdots) :=\{ \tau \in \cT : \forall \; \text{open}\; U
\ni \tau, \Lambda(\sfx_1\cdots \sfx_n) \in U \; \text{for infinitely many }n\}.\]
We define the \emph{coding tree} $T(\sfX)$ to be the abstract rooted uniformly branching simplicial tree whose edges are labelled with elements of $\sfX$.
By iteratively lifting and concatenating representing paths under $\psi$, we obtain a
continuous extension of the coding map to the coding tree, $\Lambda: T(\sfX) \to \cT$.  An edge-path with label sequence
$\sfx_1, \ldots, \sfx_n$  maps continuously to the path
$\sfx_1 \cdots  \sfx_n$ in $\cT$, which terminates at $\sfx_1
\cdots \sfx_n.\star$. The corresponding image path in $\cM$ we denote by $\pi(\sfx_1 \cdots \sfx_n)$.

Some words of caution: given an infinite word $\sfx_1\sfx_2\cdots$, the sequence of points
$\sfx_1\cdot \sfx_2\cdot \cdots \cdot \sfx_n.\star$ is not an ``orbit'' in
the traditional sense: the $(n+1)$st term is \emph{not} the image of
the $n$th term under application of the action of one of the
$\sfx_i$'s.  Thus, we do not know how to prove that if the sequence corresponding to an aperiodic infinite word recurs infinitely often to a compact subset of moduli space, then its limit under the coding map exists. 

Clearly $\pi(\Lambda(\sfX^{+\infty})) \subset \cJ$. Conversely, any $\mu \in \cJ$
arises as an element of some $\pi(\Lambda(\sfx_1\sfx_2\ldots))$, since by diagonalization, in any finite branching infinite rooted tree, any infinite sequence of vertices contains a subsequence lying along an infinite ray. So in fact $\cJ = \pi(\Lambda(\sfX^{+\infty}))$.  

We call $\Lambda(\sfX^{+\infty})$ the \emph{limit set} of the IFS on $\cT$ generated by $\sfX$. 

\begin{prop}
\label{prop:coding}
Suppose $B$ is a non-exceptional Hurwitz biset.  The following are equivalent.
\begin{enumerate}
\item The limit set $\cJ$ of the correspondence on moduli space $\phi, \rho: \cW \to \cM$ is a nonempty compact subset of $\cM$.  Equivalently, the IFS on $\cT$ generated by $\sfX$ is contracting. 
\item The coding map $\Lambda: \sfX^{+\infty} \to \cT$ is single-valued and continuous. 
\item There is a nonempty smooth compact manifold with boundary  $K \subset \cM $ such that $\psi^{-1}(K) \subset K$ and the inclusion $K \hookrightarrow \cM$ induces a surjection on fundamental groups.
\end{enumerate}
\end{prop}

\begin{proof}That (2) implies (1) follows from the fact that $\pi$ maps the limit set in $\cT$ surjectively to the limit set in $\cM$. 

Suppose that (1) holds.  Since $\pi_1(\cM)={\op \cG}$ is finitely
generated, some sufficiently large closed metric neighborhood $K$ of
$\cJ$ will carry the fundamental group of $\cM$.
Statement (2) of Proposition~\ref{prop:limitset} implies that
$\psi^{-1}(\cJ)\subset \cJ$,
so $\psi^{-1}(K) \subset K$ because $\psi^{-1}$ is distance
nonincreasing. 
Hence (1) implies (3). 

To see that (3) implies (1), we note that by enlarging $K$ if
needed we may assume that it contains each of the paths $\sfx \in
\sfX$ corresponding to basis elements, while still ensuring
$\psi^{-1}(K) \subset K$.  Proposition \ref{prop:strict} implies that
$\psi^{-p}$ is a strict  contraction, where $p:=\#P$.  Hence it is
uniformly contracting on $K$.  Then standard arguments show that the
image in $\cM$ of each infinite path in the coding tree $T(\sfX)$ regarded as a path in $\cM$ is
uniformly bounded by the sum of a convergent geometric series. The same is true when lifted to $\cT$. 

\end{proof}

As a corollary, we see that contraction of the IFS on Teichm\"uller space defined by a chosen basis is a property that is independent of the choice of basis.

We are now ready to prove another of the implications in Theorem \ref{thm:main}. \\

\emph{Proof that (2) implies (1):}  This is standard, a special case of a much more general fact that the biset over the fundamental group of any expanding (meaning, lifts of paths are uniformly contracted) correspondence is contracting.  We adopt the setup of \S \ref{subsecn:codingtree}.  Fix a backward invariant compact subset $K \subset \cM, \cJ \subset K$ as in the setup of Proposition \ref{prop:coding}.  
We denote by $|\sfx|$ the length of the chosen representative
for $\sfx$ and by $\xi:=\max\{|\sfx|: \sfx \in \sfX\}$. The hypothesis
implies that any  $g\in {\op \cG}$ is represented by a rectifiable loop in $K$ at the basepoint $\star$.  We denote by $|g|$ the minimum length of such a representative.  Since ${\op \cG}$ acts cocompactly and properly discontinuously on $\pi^{-1}(K)$, the resulting norm on ${\op \cG}$ is proper: balls are finite.  

By hypothesis, $K$ is backward-invariant.  Applying Proposition \ref{prop:strict} and passing to the $p$th
iterate, we may reduce to the case $p=1$, and may assume that lifting
of paths in $K$ under $\psi$  contracts lengths by a definite factor
$\lambda < 1$.   The definitions of taking states and interpretation in terms of paths
then imply that for any $g \in {\op \cG}$ and any $\sfx \in \sfX$, the
element $\sfx\cdot g$ is represented by a path of length at most
$\lambda|g|+\xi$. Hence if $|g|>2\xi/(1-\lambda)$, then $\sfx \cdot g$ is represented by a path in $K$ of length at most $(1/2)(1+\lambda)|g|<|g|$.  In other words, the operators $g \mapsto \sfx@g, \sfx \in \sfX$ are uniformly contracting on sufficiently large elements of ${\op \cG}$.  By induction, and properness of the norm, it follows that ${\op B}=B^{\text{paths}}$ is contracting over ${\op \cG}$, hence $B$ is contracting over $\cG$.

\section{Contraction on $\cV$ implies contraction on $\cT$}
\label{secn:3to1}

We recall the statement:

\begin{nonumthm}
Suppose $B$ is a non-exceptional Hurwitz biset of Thurston maps.  Then the following are equivalent:
\begin{enumerate}
\item The IFS on $\cG$ given in \S (\ref{subsecn:G}) is contracting.
\item The IFS on $\cT$ given in \S (\ref{subsecn:T}) is contracting.
\item The IFS on $\cV$ given in \S (\ref{subsecn:V}) is contracting.
\end{enumerate}
\end{nonumthm}

\begin{proof}

The implication (2) implies (1) was shown at the end of the previous subsection, \S \ref{subsecn:codingtree}. The implication (1) implies (3) was shown in Proposition~\ref{prop:coeffs}.
\\

\emph{Proof that (3) implies (2).} Given $\mu \in \cM$ we denote by $\ell(\mu)$ the length of the
shortest closed hyperbolic geodesic on the corresponding punctured sphere, with its hyperbolic metric; similarly we define $\ell(\tau)=\ell(\pi(\tau))$. 

Suppose that the IFS on $\cV$ is contracting, that is,
$\whatsigma(L[B]) < 1$. Proposition~\ref{prop:coding} implies that to
prove that the IFS on $\cT$ is contracting, it suffices to prove that
there is a nonempty smooth compact manifold with boundary $K\subset
\cM$ such that $K \hookrightarrow \cM$ induces a surjection on
fundamental groups.  We will show the existence of such a compact,
backward-invariant set $K$.

We equip $\cV$ with now the $\ell^\infty$ norm, denoted $|\cdot |$,
and use this to define the operator norm $||w||$ of an element $w \in
L[B^*]$.  Since $\whatsigma(L[B])<1$, there is an iterate $q$ such
that $||w||<1/2$ for all $w\in L[B^n]$ with $n\geq q$. 

In this paragraph we show how to reduce to the case in which $q=1$.
Suppose that we have a compact set $K'$ relative to $\psi^q$ instead
of $\psi$.  Put $K'_0:=K', K'_{i+1}:=\psi^{-1}(K'_i)$ for $i=0,
\ldots, q-1$, and $K:=\cup_{i=0}^q K_i'$.  Then $\psi^{-1}(K) \subset
K$.  Hence by replacing $f$ by $f^{\circ q}$, we may assume $q=1$ to
ease notation.

\begin{lem}
\label{lemma:improves}
There exists $\epsilon>0$ such that if $\ell(\mu)<\epsilon$ and $\widetilde{\mu}\in \psi^{-1}(\mu)$, then $\ell(\wtmu)>\ell(\mu)$.
\end{lem}

Assuming Lemma \ref{lemma:improves}, we now show the
implication. Mumford's compactness theorem implies that the locus $\{\ell(\mu)\geq \epsilon\} \subset \cM$ is compact.  Since $\phi$ is a finite cover, it is proper, so the locus $\psi^{-1}(\{\ell(\mu)\geq \epsilon\})$ is compact. Again by Mumford's theorem, there exists $\epsilon'>0$ such that $\psi^{-1}(\{\ell(\mu)\geq \epsilon\}) \subset \{\ell(\mu)\geq \epsilon'\}=:K$, which is compact.  If $\epsilon' \geq \epsilon$ we are done. Otherwise, the compact set $K$ is the union of the loci $\{\ell(\mu) \geq \epsilon\}$ and $\{\epsilon' \leq \ell(\mu) \leq \epsilon\}$. Under $\psi^{-1}$, the image of the former lies in $K$ by definition.  The lemma implies that the image of the latter is contained in the locus $\{\ell(\mu)\geq \epsilon'\}= K$.  So $\psi^{-1}(K) \subset K$.

Turning to the proof of the lemma, we first establish some
notation. Suppose $\mu, \wtmu$ are as in the statement.  Fix $\tau \in
\cT$ such that $\pi(\tau)=\mu$.  Given $\gamma \in \curves$, we denote
by $\ell(\tau, \gamma)$ the length of the unique geodesic homotopic to
$\gamma$ relative to the hyperbolic metric determined by $\tau$.  

Let $\Gamma$ be a nonempty multicurve.  The ${\op \cG}$-orbit of
multicurves containing $\Gamma$ is represented by $\Gamma_{\sft}$ for some
$\Gamma_{\sft}\in \sfT$.  Every $\gamma\in \Gamma$ determines a unique basis
vector $v_\gamma$ in the summand $\R[\Gamma_{\sft}]$ of $\cV$.  We set
\[Z(\Gamma, \tau):= \sum_{\gamma\in \Gamma} (1/\ell(\tau,
\gamma))v_\gamma \in \cV.\]

We fix a Margulis-like constant $0<L<2\log(\sqrt{2}+1)$; for
concreteness, $L:=\log(\sqrt{2}+1)$ will do.  Curves of length less
than $L$ are ``short''.  We set $\Gamma:=\{\gamma\in \curves :
\ell(\tau,\gamma)<L\}$, the set of all short curves.  This is a
multicurve by \cite[Corollary 3.8.7]{MR3675959}, possibly empty.  We
denote by
\[ U(\Gamma):=\{ \tau'\in \cT : \ell(\tau',\beta) \geq L\;\;  
\forall \beta \in \curves , \; \beta \not\in \Gamma\}.\]
In other words, $U(\Gamma)$ is the locus in $\cT$ for which each short
curve belongs to $\Gamma$.  Note that if $\ell(\tau)<L$, then $\Gamma$
is nonempty and $\tau \in U(\Gamma)$.

We now prove the lemma.  There exists $\sfx \in \sfX$ so that $\wttau:=\sigma_\sfx(\tau)$ satisfies $\pi(\wttau)=\wtmu$.  We denote $\sfx$ by $f$ and write $\sigma_f$ instead.   

We next show that if $\gamma\in \curves$ and $\ell
(\widetilde{\tau},\gamma)<L$, then $\gamma\in f^{-1}(\Gamma)$.  For
this we use $f$ to lift the hyperbolic metric on $S^2-P$ to a
hyperbolic metric on $S^2-f^{-1}(P)$.  Let $\delta$ be the unique
geodesic relative to this metric on $S^2-f^{-1}(P)$ which is isotopic
to $\gamma$.  Then $f(\delta)$ is a closed geodesic in $S^2-P$
isotopic to $f(\gamma)$, possibly with self-intersections.  So $\ell
(\tau,f(\delta))\le\ell (\tau,f(\gamma))<\ell
(\widetilde{\tau},\gamma)<L$.  Hence \cite[Proposition
10.11.5]{MR3675959} implies that $f(\delta)$ is simple.  Thus
$f(\delta)\in \Gamma$, and so $\gamma\in f^{-1}(\Gamma)$.

Now we apply the estimates in the proof of
\cite[Prop. 7.4]{selinger:pullback}.  Though stated for invariant
multicurves with leading eigenvalues less than 1, these assumptions are
not used to prove these estimates.  The key observation is the existence of a positive constant $C=C(\deg(f), L, \#P)$, depending only on the degree of $f$, on $L$, and on the cardinality of $P$, such that 
\[ \tau \in U(\Gamma) \implies |Z(f^{-1}(\Gamma), \wttau)| \leq||L[f,\Gamma]||\cdot | Z(\Gamma, \tau)| +C.\]
We have passed to a sufficient iterate so that $||L[f,\Gamma]|| <1/2$.  Thus 
\[ |Z(\Gamma,\tau)|>2C \implies  |Z(f^{-1}(\Gamma), \wttau)| < (1/2) |Z(\Gamma,\tau)|+(1/2) |Z(\Gamma,\tau)|=|Z(\Gamma,\tau)|.\]
In other words, since we are dealing with the $\ell^\infty$ norm, 
\[ \tau \in U(\Gamma) \; \text{and }\; \ell(\tau) < \frac{1}{2C} \implies \ell(\wttau)>\ell(\tau).\]
Taking $\epsilon:=\min\{1/(2C), L\}$, we conclude that if
$\ell(\tau)<\epsilon$, then the set of short curves on the surface
corresponding to $\tau$ is nonempty, and the estimates show that
the length of the systole on $\wttau$ cannot decrease.  The proof is complete.

\end{proof}

By Theorem \ref{thm:main}, a Hurwitz biset $B$ for which $B^*$
consists entirely of rational maps, but which is not contracting, must
have the property that the limit set $\cJ$ is not compact.
Proposition \ref{prop:limitset} implies that then the set of periodic
points in $\cJ$ is unbounded. Proposition \ref{prop:coding} and the
interpretation of the proof of Thurston's characterization in terms of
iterated concatenation of paths implies that then every periodic
coding ray terminates at such a periodic point.

\section{Examples}
\label{sec:examples}

Here, we present several examples of computations of the strata
scrambler and related results for polynomials and NET maps.

\subsection{Polynomials}
\label{subsecn:toppolys}

The main result of this section is

\begin{thm}
\label{thm:Lforpoly}
Suppose $f: (\rs, P) \to (\rs, P)$ is a topological polynomial,
$\Gamma$ is a multicurve in $\rs - P$, and $f^{-1}(\Gamma)=\Gamma$. Then $\sigma(L[f,\Gamma]) \leq 1$, with equality if and only if $\Gamma$ contains a degenerate Levy cycle.  Furthermore, if $f$ is equivalent to a complex polynomial, then $\sigma(L[f,\Gamma]) < 2^{-1/\#P}<1$.
\end{thm} 

A \emph{degenerate Levy cycle} is a Levy cycle of nontrivial curves,
each mapping by degree one, and bounding a pairwise disjoint
collection of closed discs also mapping by degree one.  Thus, in
particular, the only obstructions are Levy cycles, and in the absence
of such obstructions, the spectral radii of the transformation
$L[f,\Gamma]$ are bounded above away from $1$ by a quantity which
depends only on the size of $P$ and not on the degree of $f$. From
this and Proposition \ref{prop:polygamma} below, we easily derive
Theorem \ref{thm:toppoly} as follows.

\begin{proof}[Proof of Theorem \ref{thm:toppoly}] We will show that if
$a_1, \ldots, a_n \in L[B]$ are the weights along an arbitrary
$n$-cycle in $\cS[B]$, then $\sigma(a_n \circ \cdots \circ a_1) <
2^{-1/\#P}$.  So assume such $a_i$'s are given; we may assume $a_i =
L[f_i,\Gamma_i]$ where $f_i^{-1}(\Gamma_i)=\Gamma_{i+1}$ and
$f_n^{-1}(\Gamma_n)=\Gamma_1$.  Let $F:=f_1 \circ \cdots \circ f_n$.
Then $\Gamma:=\Gamma_1$ is a completely invariant multicurve for $F$.
Statement (6) of Proposition \ref{prop:polygamma} shows that under
either condition of Theorem \ref{thm:toppoly}, there cannot exist a
degenerate Levy cycle, and so $f$ is equivalent to a complex
polynomial.  Theorem \ref{thm:Lforpoly} then implies that
$\sigma(L[F,\Gamma]) = \sigma(a_n \circ \cdots \circ a_1) < 2^{-1/\#P}
< 1$ independent of $n$.  Thus $\widehat{\sigma}(B) < 1$ and so $B$ is
contracting by Theorem \ref{thm:main}.
\end{proof}

The proof of Theorem \ref{thm:Lforpoly} proceeds by showing that for a topological polynomial $f$ and a completely invariant multicurve $\Gamma$, the structure of the linear map $L[f,\Gamma]$ is constrained; see Proposition \ref{prop:polygamma}.  While this seems to be well-known, we were unable to find this version in the literature.

We begin with some notation and preparatory remarks.  Let $f\co
(\rs,P)\to (\rs,P)$ be a topological polynomial.

If $\gamma \subset \rs-P$ is a simple closed curve, we denote by
$D(\gamma)$ the bounded component of its complement, i.e. its
``inside''. If $\gamma, \wtgamma \in \curves$ and $\gamma \leftarrow
\wtgamma$ then up to isotopy $f: D(\wtgamma) \to D(\gamma)$ is a
proper--in particular, surjective--branched cover of degree given by
$\deg(f: \wtgamma \to \gamma)$. 

Next, we state some facts about multicurves in the plane. Suppose
$\Gamma$ is a multicurve in $\rs- P$.  If $\alpha, \beta \in \Gamma$,
then either their insides are disjoint, or one's inside is contained
in the other's.  Thus any $\Gamma$ has a nonempty distinguished subset
$I:=I(\Gamma)$ consisting of its ``innermost'' elements, defined by
the property \[ \alpha \in I \iff \forall \beta \in \Gamma -
\{\alpha\}, \; D(\alpha) \not\supset \beta .\] Moreover, the insides
of innermost curves are pairwise disjoint, i.e. $\gamma_1, \gamma_2
\in I \implies D(\gamma_1) \cap D(\gamma_2) = \emptyset$.

Next, suppose that $\gamma \leftarrow \wtgamma$. In this case, we say that $\gamma$ is a \emph{parent} of $\wtgamma$. 

Finally, suppose $\delta_1, \delta_2 \subset f^{-1}(\gamma)$ are distinct and
nontrivial preimages of a curve $\gamma$. Then their insides must be
pairwise disjoint, and so they cannot be homotopic in $\rs - P$.  That
is,
  \begin{equation*}\tag{**}
\parbox{.75\linewidth}{
  any entry of the matrix $L[f,\Gamma]$ has at most one term in
its sum, and hence is either 0 or the reciprocal of a positive integer.}
  \end{equation*}

We introduce some notation. We put $L:=L[f,\Gamma]: \R[\Gamma] \to \R[f^{-1}(\Gamma)]$.  For a subset $S \subset f^{-1}(\Gamma)$ we put $L_S=p_S \circ L$ where $p_S: \R[f^{-1}(\Gamma)] \to \R[S]$ is the natural projection.  Matrices--after a possible permutation of the basis elements $\Gamma$--are denoted by $M(\cdot)$.

\begin{prop}
\label{prop:polygamma}
Suppose $f$ is a topological polynomial, and $\Gamma$ is an arbitrary multicurve such that $f^{-1}(\Gamma) \supset \Gamma$.  Let $I:=I(\Gamma)$ be the innermost elements of $\Gamma$. 
\begin{enumerate}
\item $(f^{-1}(I)) \cap \Gamma \subset I$.  That is, a preimage of a $\Gamma$-innermost curve that belongs to $\Gamma$ is again $\Gamma$-innermost, and so  $L_\Gamma$ leaves invariant the subspace $\R[I]$, i.e. we may assume that the matrix for $L_\Gamma$ has the form 
\[ 
M(L_\Gamma) = \begin{pmatrix}
M(L_I) & * \\
0 & *
\end{pmatrix}.
\]
\item For each $\wtgamma \in I$, there exists at most one ``parent'' $\gamma \in I$ with $\gamma \leftarrow \wtgamma$.  Thus pushforward under $f$ induces a function $f_*: I \cup \{\emptyset\} \to I \cup \{\emptyset\}$ upon setting $f(\emptyset)=\emptyset$ and $f(\wtgamma)=\emptyset$ if $\wtgamma \in I$ has no parent in $I$. 

The grand orbits of $f_*$ partition $I$ into the pairwise disjoint union of an escaping set $E \subset I$, consisting of those curves which land on $\emptyset$ under iteration of $f_*$, and finitely many sets $K \subset I$ consisting of elements iterating to a single common cycle, $C(K)$. Thus $L_I$ leaves invariant the subspace $\R[E]$ on which the action of $L_I$ is nilpotent, and $L_I$ leaves invariant each of the subspaces $\R[K]$. Thus
\[ M(L_I)=
\begin{pmatrix}
M(L_{K_1}) & 0 & \cdots  & 0\\
0  & M(L_{K_2}) & \cdots  & 0\\
\cdots  &  \cdots & \cdots & \cdots  \\
0 & 0 & \cdots &  M(L_E)
\end{pmatrix}.
\]

\item Each set $K$ decomposes into its unique cycle $C(K)$ and its complement $K-C(K)$. So 
\[ M(L_K)=
\begin{pmatrix}
M(L_{K-C(K)}) & * \\
0 & M(L_{C(K)})\\
\end{pmatrix}
\]
where $M(L_{K-C(K)})$ is nilpotent.
\item Writing $C=c_1 \leftarrow c_2 \leftarrow \cdots c_{p} \leftarrow c_1$, we have 
\[
M(L_C)=
\begin{pmatrix}
0 & 0 & 0 & \cdots & 0 & 1/d(c_{p})\\
1/d(c_1) & 0 & 0 & \cdots & 0 & 0 \\
0 & 1/d(c_2)& 0 & \cdots & 0 & 0  \\
0 & 0 & 1/d(c_3)& \cdots & 0 & 0 \\
\cdots & \cdots & \cdots & \cdots & \cdots & \cdots \\
0 & 0 & 0 & \cdots & 1/d(c_{p-1}) & 0 \\
\end{pmatrix}
\]
That is, $M(L_C)=\Pi W$, where $\Pi$ is a transitive permutation matrix, and $W$ is a diagonal matrix whose diagonal entries have the form $W_{c}=1/d(c), c \in C$ where $d(c)=\deg(f: c \to f(c)) \in \{1, \ldots, \deg(f)\}$.  Thus $\sigma(L_C) = (\prod_{c\in C}d(c))^{-1/\#C}$.  
\item The insides of the curves comprising a cycle $C$ are pairwise disjoint disks which are permuted up to isotopy by the action of $f$, and each such disk contains at least one periodic point of $P$.
\item If either (i) each cycle of $P$ contains a critical point, or (ii) If there is a unique cycle in $P$ not containing a critical point, and this cycle has length $1$, then for each cycle $C\subset \Gamma$, we have $d(c) \leq 1/2$ for some $c \in C$, and so $\sigma(M(L_C))\leq 2^{-1/\#C}<2^{-1/\#P}$. 

\end{enumerate}
\end{prop}

Assuming this proposition, Theorem \ref{thm:Lforpoly} follows easily.
We inductively decompose $\Gamma$ by partitioning it into innermost
and non-innermost elements, setting $\Gamma^0:=\Gamma$ and
$\Gamma^{r+1}:=\Gamma^r - I(\Gamma^r)$.  The inductive assumption that
$f^{-1}(\zG^r)\supset \zG^r $ and the conclusion from (1) that
$(f^{-1}(I(\zG^r))) \cap \Gamma^r \subset I(\zG^r)$ imply that
$f^{-1}(\Gamma^{r+1}) \supset \Gamma^{r+1}$.  This allows us to
proceed inductively.  At each stage $r$, we apply conclusion (1) of
the proposition to conclude that $\sigma(L_{\Gamma^r}) \leq
\max\{\sigma(L_{I(\Gamma^r)}), \sigma(L_{\Gamma^{r+1}})\}$. The
remaining conclusions of the proposition bound
$\sigma(L_{I(\Gamma^r)})$.

We continue with the proof of the proposition. 

(1) Suppose to the contrary that for some $\gamma \in I$ and $\wtgamma
\in \Gamma - I$ we have $\gamma \leftarrow \wtgamma$. Then there
exists $\alpha \in \Gamma$ and $\wtalpha \in \Gamma$ with $\alpha
\leftarrow \wtalpha$ and $\wtalpha \subset D(\wtgamma)$.  But then up
to isotopy $\alpha = f(\wtalpha) \subset f(D(\wtgamma)) = D(\gamma)$.
However, $\za\ne \zg$ because the insides of distinct preimages of
$\zg$ must be disjoint.  This contradicts the assumption that $\gamma
\in I$.

(2) Two distinct elements of $I$ bound disjoint disks.  Therefore, any two components of the preimages of these disks are also disjoint. Their boundaries, if essential, cannot then be homotopic in $\rs - P$. So any curve in $I$ has at most one parent in $I$.

(3) Immediate from (2).

(4) Immediate from (3) and observation (**) above.   

(5) The elements of $C$ are in $I$, thus they are nontrivial, and their insides are pairwise disjoint and must contain elements of $P$, and $f$ maps insides to insides. 

(6) 
Since every innermost disk contains at least two postcritical points,
the assumptions imply that every innermost disk contains a
postcritical point which lies in a cycle containing a critical point.
Therefore every inner disk contains a postcritical point which lies in
a cycle containing a critical point.  Statement (6) easily follows
from this and observation (**).

This concludes the proof of Proposition \ref{prop:polygamma} and Theorem \ref{thm:Lforpoly}

\subsection{A polynomial example with $\#P=5$}
\label{subsecn:julia4cycle}

A coarse invariant of a Hurwitz biset containing a Thurston map $f$ is
its underlying \emph{critical orbit portrait}--the directed weighted
graph with vertex set $P_f \cup C_f$ and an edge joining $z$ to
$f(z)$ weighted by the local degree $\deg(f,z)$.  We think of this as
defining a self-map $\tau: P_f \cup C_f \to P_f$. Floyd \emph{et
al.} \cite{floyd2022realizingpolynomialportraits} have characterized
those polynomial portraits which are \emph{completely unobstructed} in
the sense that any topological polynomial with such a portrait is
equivalent to a complex polynomial:

\begin{thm}
\label{thm:polyportrait}
Suppose $\Pi$ is an abstract polynomial portrait. Then every
Thurston map with portrait isomorphic to $\Pi$ is unobstructed if and only if $\Pi$ 
satisfies at least one of the following conditions.
\begin{enumerate}
\item $\Pi$ has at most three postcritical vertices.
\item Every cycle of $\Pi$ is an attractor (contains a critical point).
\item $\Pi$ has a single non-attractor cycle, and it has length one.
\item Every finite postcritical vertex of $\Pi$ is in a single non-attractor cycle,
this cycle has length $p^k$ 
for some prime number $p$ and some positive
integer $k$, the finite postcritical vertices can be enumerated as
$\{v_i: 0 \leq  i < p^k\}$ such that $\tau(v_i) = v_{i+1 \bmod p^k}$ for every $i\in \{0, \ldots, p^k-1\}$, and 
if $v_j$ is a critical value, then $j$ is a multiple of $p^{k-1}$.
\end{enumerate}
\end{thm}

In the first case of $\#P=3$, the group $\cG$ is trivial, so the Hurwitz biset reduces to a singleton. In cases 2 and 3, the biset is contracting, by Theorem \ref{thm:toppoly}.  A portrait can be iterated in the obvious way.  It is easy to see that the portraits in cases 1, 2, 3 are closed under iteration, but that those in case 4 are not. 
Below, we show by example that there exists a polynomial $f$ whose portrait is covered by case 4--so that $B(f)$ consists entirely of rational maps--but for which $B(f)^2$ contains obstructed maps. This naturally raises the following \\

\noindent{\bf Question:} \emph{Suppose $f$ is a topological polynomial whose portrait satisfies case 4.  When does $B(f)^*$ contain an obstructed element? }\\

Our treatment of the example also illustrates the theme that for
low-complexity polynomial examples with $\#P>4$, it is possible to
make by-hand computations of the strata scrambler.

\begin{ex}\label{ex:cubicpoly} A cubic polynomial with $\#P=5$
\end{ex}
Consider the following portrait, of type case 4, corresponding to a
topological polynomial of degree $3$: \[ \infty \stackrel{3} {\mapsto}
\infty\] \[ c_1 \stackrel{2}{\mapsto}a\;\;\;\;c_2
\stackrel{2}{\mapsto}c\] \[ a \stackrel{1}{\mapsto} b
\stackrel{1}{\mapsto} c \stackrel{1}{\mapsto} d \stackrel{1}{\mapsto}
a.\] It is easy to find topological and complex polynomials realizing
this portrait; we choose one, $f$, put $P:=P_f$, and write
$P:=\{a,b,c,d,\infty\}$ and put $\cG:=\PMod(S^2, P)$ as usual.  The
restriction $f: \C - f^{-1}(\{a,c\}) \to \C-\{a,c\}$ is a degree 3
unramified non-cyclic covering, and is therefore unique up to
isomorphism of covering spaces, and has trivial deck group. It follows
that any homeomorphism $g$ representing an element of $\cG$ lifts
under $f$ to a homeomorphism $\widetilde{g}: (S^2, f^{-1}(P)) \to
(S^2, f^{-1}(P))$ which must preserve the fibers of $f$ and the
property of being a critical point. It follows that $\widetilde{g}$
fixes each element of $f^{-1}(\{a,c\})$, but may permute the elements
of $f^{-1}(b)$ and of $f^{-1}(d)$. If $g$ is a Dehn twist about
$\{a,b\}$, then $\widetilde{g}$ acts as a transposition on $f^{-1}(b)$
and as the identity on $f^{-1}(d)$, and a Dehn twist about $\{b,c\}$
acts as a different transposition and as the identity on
$f^{-1}(d)$. By symmetry Dehn twists about $\{c,d\}$ and $\{d,a\}$ act
as different transpositions on $f^{-1}(d)$ and as the identity on
$f^{-1}(b)$.  It follows the induced map $\cG \to
\text{Sym}(f^{-1}(b)) \times \text{Sym}(f^{-1}(d))$ is surjective.
From this we conclude that there is a unique Hurwitz biset $B$ with
this portrait; this could also have been deduced by computing the
correspondence on moduli space (which turns out to be a map), and
noting that the computation depends only on the portrait. We also note
that $f^{-1}(b)\cap P=\{a\}$, $f^{-1}(d)\cap P=\{c\}$ and the
stabilizer of $\{a,c\}$ in $\text{Sym}(f^{-1}(b)) \times
\text{Sym}(f^{-1}(d))$ is isomorphic to $(\Z/2\Z)\times (\Z/2\Z)$.
From this we conclude that $[\cG:\cG_f]=6^2/2^2=9$.

In order to compute the strata scrambler of the Hurwitz biset
represented by $f$, we need some notation.  When $\#P=5$ every simple
closed curve $\gamma$ which is neither peripheral nor null homotopic
separates two postcritical points, say $x$ and $y$, from the other
three.  This gives a bijection $\gamma.\cG \mapsto \{x,y\}$ from
$\cG$-orbits of such curves to the set of two-element subsets of $P$.
We denote the $\cG$-orbit of the curve $\zg$ by $xy$, with the convention $xy=yx$.  We denote the
$\cG$-orbit of a 2-component multicurve with curve orbits $wx$ and $yz$
by $wx\&yz$.  We denote the $\cG$-orbits of simple closed curves which
are either peripheral or null homotopic by $\odot $.  We also use this
notation to denote the vertices of our strata scrambler.  We will
indicate weights of strata scrambler edges simply by the matrices of
the linear transformations.

Describing the strata scrambler amounts to describing all ways in
which $f$ pulls back multicurves, up to the action of $\cG$.  Equivalently, it amounts to all
ways in which the twists of $f$ pullback a fixed set of orbit
representatives, up to the action of $\cG$.  This is what we compute.  So let $g$ be a $\cG$-twist
of $f$.  

We begin with the orbit $bd$.  We choose a proper arc $\za$ for
$(S^2,P)$ with endpoints $b$ and $d$.  We view $\za$ as a core arc for
a simple closed curve $\zg$, and we consider the $g$-pullback of this
simple closed curve.  Since neither $b$ nor $d$ is a critical value,
$g^{-1}(\za)$ consists of three disjoint arcs.  The union of these
arcs contains $a$ and $c$ and no other postcritical point.  Since $\cG
\to \text{Sym}(f^{-1}(b)) \times \text{Sym}(f^{-1}(d))$ is surjective,
it is possible to choose $g$ so that $a$ and $c$ are contained in the
same connected component of $g^{-1}(\za)$ and it is possible to choose
$g$ so that $a$ and $c$ are not contained in the same connected
component of $g^{-1}(\za)$.  So our scrambler has an edge from $bd$ to
$ac$ and an edge from $bd$ to $\odot $ and no other edges with
initial vertex $bd$.  The edge from $bd$ to $ac$ has label $[1]$ (the
$1\times 1$ matrix with entry 1) because the nontrivial connected
component of $g^{-1}(\zg)$ when there is one maps to $\zg$ with degree
1.

Next consider the orbit $ac$.  Now a core arc for $\zg$ is a proper
arc $\za$ for $(S^2,P)$ with endpoints $a$ and $c$.  Two lifts of
$\za$ have endpoint $c_1$, the critical point of $g$ which maps to
$a$.  The other two endpoints of these lifts cannot both be $c_2$, the
critical point which maps to $c$, because every connected component of
the bounded disk $D(\zg)$ of $\zg$ is a topological disk.  So one lift
of $\za$ has endpoints $d$ and $c_1$, another lift has endpoints $c_1$
and $c_2$ and the third lift has endpoints $c_2$ and $b$.  Hence
$g^{-1}(\za)$ is homeomorphic to a line segment.  It contains exactly
two postcritical points, namely $b$ and $d$, and it maps to $\za$ with
degree 3.  Our scrambler has exactly one edge with initial vertex
$ac$.  This edge ends at $bd$ and its weight is $[1/3]$.

Next consider the orbit $ab$.  Two lifts of $g^{-1}(\za)$ have the
critical point $c_1$ in common and contain two preimages of $b$.  The third
lift is disjoint from them.  It contains $d$ and one preimage of $b$.
Using the surjectivity of $\cG \to \text{Sym}(f^{-1}(b)) \times
\text{Sym}(f^{-1}(d))$, we find that two edges of our scrambler begin
at $ab$.  One of these ends at $\odot$.  The other ends at $ad$ with
weight $[1]$.

We determine edges beginning at $ad$, $bc$ and $cd$ as we did for
$ab$.  Edges begin at $ad$, $bc$ and $cd$ and end at $\odot $.  Edges
with weights $[1]$ begin at $ad$, $bc$, respectively, $cd$ and end at
$cd$, $ab$, respectively, $bc$.

Next consider the orbit $b\infty $.  In this case all three $g$-lifts
of $\za$ meet at $\infty $ and end at the three elements of
$g^{-1}(b)$.  So $g^{-1}(\alpha)$ contains $\infty $, $a$ and no other
postcritical point.  Our scrambler has exactly one edge with initial
endpoints $b\infty $.  It ends at $a\infty $ and it has weight $[1/3]$.

Similarly, our scrambler has exactly one edge with initial endpoint
$d\infty $.  It ends at $c\infty$ and it has weight $[1/3]$.

Next consider the orbit $a\infty $.  The three $g$-lifts of $\za$ meet
at $\infty $ and two of them also meet at $c_1$.  So $S^2-g^{-1}(\za)$
consists of two open topological disks $D_1$, respectively $D_2$,
mapping to $S^2-\za$ with degree 1, respectively degree 2.  The set
$g^{-1}(\za)$ contains $d$ and $\infty $, and the disk $D_1$ contains
$b$.  There are essentially four possibilities for the locations of
the postcritical points $a$ and $c$.  One possibility is that $a,c\in
D_1$.  In this case the connected component of $g^{-1}(\zg)$ in $D_1$
separates $a$, $b$ and $c$ from $d$ and $\infty $.  The other
connected component of $g^{-1}(\zg)$ is null homotopic.  We obtain a
scrambler edge from $a\infty $ to $d\infty $ with weight $[1]$.  If
$a\in D_1$ and $c\in D_2$, then the connected component in $D_1$
separates $a$ and $b$ from $c$, $d$ and $\infty $, while the other
connected component is peripheral.  We obtain an edge from $a\infty $
to $ab$ with weight $[1]$.  If $a\in D_2$ and $c\in D_1$, then we
obtain an edge from $a\infty $ to $bc$ with weight $[1]$.  If $a,c\in
D_2$, then we obtain an edge from $a\infty $ to $ac$ with weight
$[1/2]$.

The situation for orbit $c\infty $ is similar to that for $a\infty $.
We obtain scrambler edges from $c\infty $ to $b\infty $, $cd$, $ad$,
respectively $ac$, with weights $[1]$, $[1]$, $[1]$, respectively
$[1/2]$.

We have so far handled all 1-component multicurves.  The 2-component
multicurves are handled in the same way using two disjoint arcs
instead of one.  Figure~\ref{fig:cubicex} succinctly describes  the
strata scrambler of this Hurwitz biset.

\begin{figure}
\centerline{\includegraphics[width=6in]{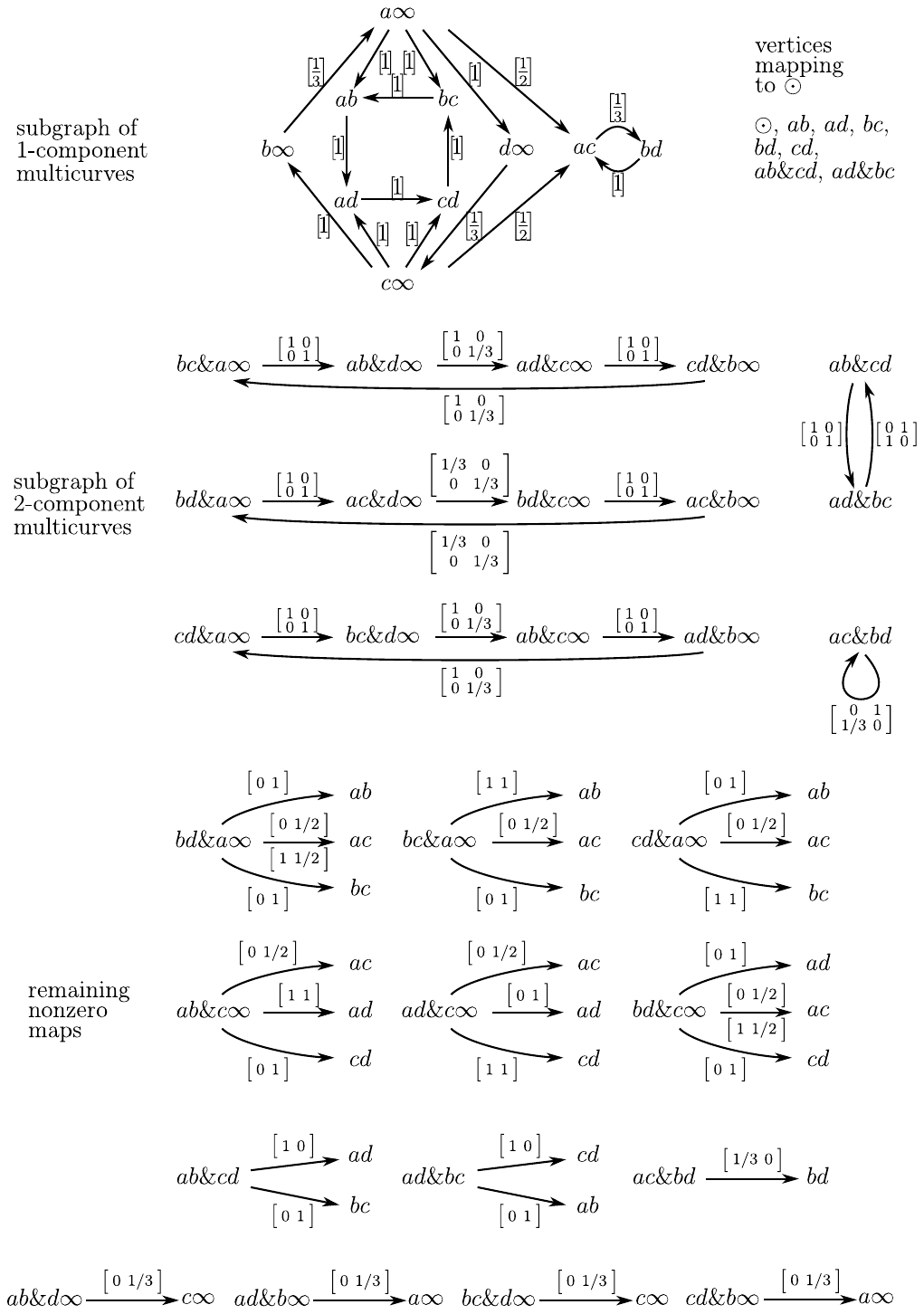}}
\caption{The scrambler for Example \ref{ex:cubicpoly}}
\label{fig:cubicex}
\end{figure}

Note the curves $ab, cd$ are disjoint, as are $bc, ad$.  The $2$-component multicurve cycle $ab\&cd \rightarrow ad\&bc \rightarrow ab\&cd$ shows that the second iterate $B^2$ of $B$ contains an obstructed map, and so $B^n$ contains an obstructed map if $n$ is even.  In particular, $B$ is not contracting. Because every cycle has even length, if $n$ is odd, then $B^n$ consists only of rational maps.

\subsection{Case $\#P=4$}

Here we specialize to the case of a Hurwitz biset $B$ for which $\#P=4$.  The space $\cT$ may be identified with the upper-half plane $\IH$ and the moduli space $\cM$ with $\Pone - \{0,1,\infty\}$ in such a way that the deck group $\cG$ of the cover $\pi: \cT \to \cM$ is identified with the principal congruence subgroup $P\Gamma(2)$ acting by linear fractional transformations.  The domain $\cW$ of the correspondence is thus a Riemann surface with a finite number of cusps.  These cusps are in bijective correspondence with orbits of cusps under the subgroup of liftables $\cG_f$, for any chosen representative $f \in B$.  In this setting, the analysis is particularly tractable. 

In the present situation $\mc =\curves$.  The set $\curves$ may be
identified with ``slopes'' of curves.  With such identifications,
given a curve $\gamma$ of slope $s$, pinching $\gamma$ so that its
hyperbolic length tends to zero produces a path in $\IH$ which limits
to a cusp point $t=-1/s$. We denote by $\IH^* = \IH \cup
\widehat{\Q}$ where $\widehat{\Q}:=\Q \cup \{\infty:=\pm 1/0\}$
as usual. A nonempty multicurve $\Gamma$ has exactly one element, and
under the action of $\cG$, there are precisely three orbits,
corresponding to whether the slope of the curve, written in lowest
terms as a rational $p/q$, takes the form $0/1, 1/0, 1/1$ when
numerator and denominator are reduced modulo $2$.

By adjoining finitely many points called cusps to $\cM$ and $\cW$,
they can be enlarged to compact Riemann surfaces $\cM^*$ and $\cW^*$.
The maps $\zeta$, $\phi$, $\rho$ and $\pi$ in Section~\ref{subsec:corr}
extend continuously to maps, $\zeta:\IH^*\to \cW^*$, $\phi: \cW^*\to
\cM^*$, $\rho:\cW^*\to \cM^*$ and $\pi:\IH^*\to \cM^*$, with $\zeta$,
$\zf$ and $\zp$ taking cusps to cusps. The map $\sigma_f: \IH \to \IH$
extends continuously to $\IH^*$ when equipped with the topology in
which horoballs tangent to a cusp form a local neighborhood basis.

Suppose now a Thurston map $f:(S^2, P) \to (S^2, P)$ is given. Under the pullback relation on curves, a nontrivial curve $\gamma$ pulls back to at most one nontrivial curve $\tilde{\gamma}$, though there may be several components.  Thus the matrices for the transformations $L[f,\Gamma]$ are one-by-one, of the form $[m]$; we call the value $m$ the \emph{multiplier} of the curve (and of the corresponding slope $s$ and cusp $t$).

The main result of this subsection, Proposition \ref{prop:cuspdegrees}, implies that these multipliers, and hence the strata scrambler, may be directly computed from the asymptotic behavior of $\phi$ and $\rho$ at the cusps.  We then illustrate this in two examples.

\begin{prop}
\label{prop:cuspdegrees} Let $t$ be a cusp of $\IH^*$. Suppose $f$ has multiplier $m$ at $t$ and set $d:=[\Stab_\cG(t):\Stab_{\cG_f}(t)]$.  Then 
\begin{enumerate}
\item $\deg(\phi,\zeta(t))=d$, and
\item $\deg(\rho, \zeta(t))=md$.
\end{enumerate}
\end{prop}
  \begin{proof} There exists a horoball $H\subseteq \IH$ at $t$ such
that if $\varphi\in \cG$ and $\sigma_\varphi(H)\cap H\ne \emptyset $,
then $\varphi\in\Stab_\cG(t)$.  It follows that the canonical map
$H/\Stab_{\cG_f}(t)\to \cW$ is a homeomorphism onto a deleted
neighborhood of $\zeta(t)$ in $\cW^*$.  Similarly, $\pi$ induces a
homeomorphism $H/\Stab_\cG(t)$ onto a deleted neighborhood of $\pi(t)$
in $\cM^*$.  Statement 1 follows.

To prove the second statement, set $t':=\sigma_f(t)$.  Choose a horoball in $\IH$ at $t'$ in
the same way that $H$ was chosen for $t$.  Because $\sigma_f$ is
continuous at $t$, we may assume that $\sigma_f(H)\subseteq H'$.  As
before, the canonical map $H'/\Stab_\cG(t')\to \cM$ is a
homeomorphism onto a deleted neighborhood of $\pi(t')$ in $\cM^*$.  We
obtain the following commutative diagram for some map
$\widetilde{\rho}$.
 \begin{equation*}
\xymatrix{H\ar[r]^{\sigma_f}\ar[d] & H'\ar[d] \\
H/\Stab_{\cG_f}(t) \ar@{^{(}->}[d] 
\ar[r]^{\widetilde{\rho}} & H'/\Stab_\cG(t') \ar@{^{(}->}[d] \\
\mathcal{W^*}\ar[r]^\rho & \cM^*} 
  \end{equation*}
The degree of $\rho$ at $\zeta(t)$ thus equals the degree of
$\widetilde{\rho}$.  We may view $\Stab_{\cG_f}(t)$ and
$\Stab_\cG(t')$ as the fundamental groups of
$H/\Stab_{\cG_f}(t)$ and $H/\Stab_\cG(t')$.  

Let $\gamma$ be a curve with slope $-1/t$, and let $\gamma'$ be a
curve with slope $-1/t'$.  Then $\gamma'=f^{-1}(\gamma)$ as
multicurves, $\Stab_\cG(t)=\text{Tw}(\gamma)$ and
$\Stab_\cG(t')=\text{Tw}(\gamma')$.  Let $\eta:=\text{tw}(\gamma)$
and $\eta':=\text{tw}(\gamma')$ be generators for $\Stab_\cG(t)$ and
$\Stab_\cG(t')$.  Then $\eta^d$ is a generator for $\Stab_{\cG_f}(t)$.
Since $[m]$ is the matrix of $L[f,\gamma]:\R[\gamma]\to \R[\gamma']$,
it follows that $\eta^d\circ f=f\circ \eta'^{md}$ and $\sigma_f\circ
\sigma_{\eta^d}=\sigma_{\eta'^{m d}}\circ \sigma_f$.  In other words,
the map on fundamental groups induced by $\widetilde{\rho}$ sends the
generator in the domain to $md$ times the generator in the
codomain. Thus the local degree of $\rho$ at $\zeta(t)$ is $m d$.

\end{proof}

Hence if $\rho, \phi$ are known explicitly, then the series expansions
at the cusps allow for the computation of the multipliers, and hence
of the scrambler. We now apply this to some examples.

\begin{ex}
\label{ex:dendrite} The dendrite and rabbit quadratic polynomials
\end{ex}

We compute here $\cS[B]$ for each of two Hurwitz
bisets $B$: that of the ``dendrite'' $f(z)=z^2+i$ and that of the
``rabbit'' $g(z)=z^2+\kappa$ where $\Im(\kappa)>0$ and $z=0$ has period $3$.  

For the dendrite, we have the postcritical set $P_f:=\{i, -1+i, -i,
\infty\}$.  Let $a, b, c$ be curves which are boundaries of small
regular neighborhoods of the Euclidean segments joining $-i$ to $+i$,
$+i$ to $-1+i$, and $-1+i$ to $-i$, respectively.  So $\mc(S^2,
P_f)\simeq \{a,b,c,\emptyset\}$. It is shown in \cite[\S
6]{bartholdi:nekrashevych:twisted} that in this case $\rho$ is
injective, and (in our notation) one may choose coordinates so that
$\cM=\rs - \{0,1,\infty\}$ and $\phi(\rho^{-1}(w))=(1-2/w)^2$, where
the cusps $0$, $\infty$, and $1$ are identified with $\cG.a, \cG.b,
\cG.c$, respectively.  Applying Proposition \ref{prop:cuspdegrees}, we
find $\cS[B]$ is given by \[[1] \lefttorightarrow c
\stackrel{[1]}{\longrightarrow} b \stackrel{[1/2]}{\longrightarrow} a
\stackrel{0}{\longrightarrow} \emptyset \righttoleftarrow 0 .\] One
can also verify this directly, using the fact that $\sfX:=\{f, f\cdot
a\}$ is a basis of $B$, applying the definition of the scrambler, and
finding preimages of the curves $a, b, c$ under $f$ and $f \cdot a$;
here we denote also by $a$ the right Dehn twist about $a$.

For the rabbit, setting $a, b, c$ to be boundaries of neighborhoods of
segments joining $0$ to $\kappa$, $\kappa$ to $\kappa^2+\kappa$, and
$\kappa^2+\kappa$ to $0$, respectively, we find similarly that $\cS
[B]$ is represented by \[ \xymatrix{ 0 \lefttorightarrow \emptyset &
\ar[l]_(.3){0} c \ar[r]^{[1]}& b \ar[r]^{[1/2]} & a
\ar@/^1pc/[ll]^{[1/2]} \\ } \]

\begin{ex}\label{ex:fixed}Critically fixed cubic and its impure twist
\end{ex}

Here we compute $\cS[B(f)]$ for the Hurwitz biset $B(f)$ determined by a cubic map $f$ with four simple fixed critical points.   Writing the most general normal form for a cubic rational map with fixed critical points at $0, 1, \infty$, and having a critical point at $x$ mapping to a critical value $y$, we see by direct calculation that $x, y$ are rational functions of a complex variable $t$--that is, the genus of $\cW$ is equal to zero--and that these are given by $x(t)=\rho(t)$ and $y(t)=\phi(t)$ where 
\[ \rho(t)=\frac{-1+2t+3t^2}{4t},\]
\[ \phi(t)=\frac{(1+t)(-1+3t)^3}{16t}.\]
Under the correspondence $\psi=\phi \circ \rho^{-1}$, each of the three cusps $a=0, b=1, c=\infty$ is fixed; each has two possible multipliers, $1$ and $1/3$; it follows that the set of cusps is completely invariant under the multivalued map $\phi \circ \rho^{-1}$.  So $\cS[B(f)]$ is represented by 
\[ 
\xymatrix{
\{[1/3], [1]\} \lefttorightarrow a & \{[1/3], [1]\} \lefttorightarrow b & \{[1/3], [1]\} \lefttorightarrow c  \\
}
.\]
Direct computation shows that there are no fixed-points in moduli space, so every element of $B(f)$ is obstructed; note however that this cannot be immediately inferred from the diagram above, without further argument. 

We now ``twist'' the preceding example with an impure mapping class
element to obtain a new example. Let $P:=P_f$.  There exists $h:
(S^2,P) \to (S^2,P)$ which is an impure homeomorphism cyclically
permuting three elements of $P$ and fixing the remaining one such that
the induced map on moduli space cyclically permutes the cusps as $c \mapsto b \mapsto a \mapsto c$.  The composition $g:=h\circ f$ is a Thurston map with $P_g=P_f$.   But now $\cS[B(g)]$ is represented by 
\[ 
\xymatrix{
a  \ar[rr]^{\{[1/3],[1]\}} && b \ar[rr]^{\{[1/3],[1]\}} && c  \ar@/^1pc/[llll]^{\{[1/3], [1]\}}
}
.\]
No element of $B(g)$ can have an invariant multicurve.  In particular, no element of $B(g)$ can have an obstruction.  Thus $B(g)$ consists entirely of unobstructed maps. 
However, the scrambler for $g^3$ now fixes each cusp, and each cusp has the set of multipliers $\{1/3^3, 1/3^2, 1/3, 1\}$.  In particular, there are obstructed elements in $B(g^3)$.

\subsection{Nearly Euclidean Thurston (NET) maps}
\label{subsecn:NETmaps}

For so-called Nearly Euclidean Thurston (NET) maps
\cites{cfpp:net, kmp:origami, Floyd:2017kq, MR4278340, Floyd:2017fj},
we have $\#P=4$, and there is an algorithm 
which 
computes the strata scrambler.
In turn, this allows one to check if the conditions in Theorem
\ref{thm:main} hold to see if the biset is contracting.  These
algorithms are implemented in the computer program {\tt NETMap}
\cite{netmapsite}.  
The NET map website \cite{netmapsite} has thousands of NET map strata
scramblers.

Quite a few NET maps have contracting bisets.  NET maps arise from
$2\times 2$ matrices of integers, and such matrices have two
elementary divisors.  Almost all NET maps with equal elementary
divisors have contracting bisets.

We explain ``almost all''.  Let $f$ be a NET map whose elementary
divisors equal the integer $n\ge 2$.  Then $\deg(f)=n^2$.  The map $f$
is a flexible Latt\`{e}s  map postcomposed with a push homeomorphism.
The push homeomorphism pushes the four postcritical points of the
flexible Latt\`{e}s  map to a set $H$ of four points which can be
viewed as elements of the finite Abelian group $A=(\bbZ/2n\bbZ)\oplus
(\bbZ/2nZ)$ modulo the action of $\{\pm 1\}$.  Any four-point subset
$H\subseteq A/\{\pm 1\}$ is possible.  Using the results of Section 4
of \cite{cfpp:net}, one can show that if the biset containing $f$ is not
contracting, then $A$ contains a cyclic subgroup $B$ of order $2n$
such that two points in $H$ lift to $B$ and the other two lift to the
unique nontrivial coset of $B$ which contains an element of order 2.
This is a strong condition.

For example, consider the case $n=4$.  So $\deg(f)=16$.  According to
\cite{netmapsite}, there are 155 mapping class group Hurwitz classes
(44HClass1-155) with elementary divisors $(4,4)$.  Here Hurwitz
equivalence allows for composition by an arbitrary element of the
mapping class group, and so these Hurwitz classes are in general
finite unions of the augmented Hurwitz classes considered in the rest
of this work.  All bisets which arise from 141 of them are
contracting.  Of these 141 Hurwitz classes, 6 have constant Thurston
pullback maps: every map in the strata scrambler is the zero map.

\bibliographystyle{amsalpha}
\bibliography{cuhrefs.bib}

\begin{bibdiv}
\begin{biblist}

\bib{MR3781419}{article}{
      author={Bartholdi, Laurent},
      author={Dudko, Dzmitry},
       title={Algorithmic aspects of branched coverings {I}/{V}. {V}an
  {K}ampen's theorem for bisets},
        date={2018},
        ISSN={1661-7207},
     journal={Groups Geom. Dyn.},
      volume={12},
      number={1},
       pages={121\ndash 172},
  url={https://mathscinet-ams-org.proxyiub.uits.iu.edu/mathscinet-getitem?mr=3781419},
      review={\MR{3781419}},
}

\bib{MR3852445}{article}{
      author={Bartholdi, Laurent},
      author={Dudko, Dzmitry},
       title={Algorithmic aspects of branched coverings {IV}/{V}. {E}xpanding
  maps},
        date={2018},
        ISSN={0002-9947},
     journal={Trans. Amer. Math. Soc.},
      volume={370},
      number={11},
       pages={7679\ndash 7714},
  url={https://mathscinet-ams-org.proxyiub.uits.iu.edu/mathscinet-getitem?mr=3852445},
      review={\MR{3852445}},
}

\bib{bartholdi:nekrashevych:twisted}{article}{
      author={Bartholdi, Laurent},
      author={Nekrashevych, Volodymyr},
       title={Thurston equivalence of topological polynomials},
        date={2006},
        ISSN={0001-5962},
     journal={Acta Math.},
      volume={197},
      number={1},
       pages={1\ndash 51},
      review={\MR{MR2285317}},
}

\bib{MR4213769}{article}{
      author={Bartholi, Laurent},
      author={Dudko, Dzmitry},
       title={Algorithmic aspects of branched coverings {II}/{V}: sphere bisets
  and decidability of {T}hurston equivalence},
        date={2021},
        ISSN={0020-9910},
     journal={Invent. Math.},
      volume={223},
      number={3},
       pages={895\ndash 994},
         url={https://doi-org.proxyiub.uits.iu.edu/10.1007/s00222-020-00995-2},
      review={\MR{4213769}},
}

\bib{MR1152485}{article}{
      author={Berger, Marc~A.},
      author={Wang, Yang},
       title={Bounded semigroups of matrices},
        date={1992},
        ISSN={0024-3795},
     journal={Linear Algebra Appl.},
      volume={166},
       pages={21\ndash 27},
  url={https://doi-org.proxyiub.uits.iu.edu/10.1016/0024-3795(92)90267-E},
      review={\MR{1152485}},
}

\bib{MR2003315}{article}{
      author={Blondel, Vincent~D.},
      author={Theys, Jacques},
      author={Vladimirov, Alexander~A.},
       title={An elementary counterexample to the finiteness conjecture},
        date={2003},
        ISSN={0895-4798},
     journal={SIAM J. Matrix Anal. Appl.},
      volume={24},
      number={4},
       pages={963\ndash 970},
         url={https://doi-org.proxyiub.uits.iu.edu/10.1137/S0895479801397846},
      review={\MR{2003315}},
}

\bib{MR1749249}{article}{
      author={Brezin, Eva},
      author={Byrne, Rosemary},
      author={Levy, Joshua},
      author={Pilgrim, Kevin~M.},
      author={Plummer, Kelly},
       title={A census of rational maps},
        date={2000},
     journal={Conformal geometry and dynamics},
      volume={4},
       pages={35\ndash 74},
  url={https://www.ams.org/journals/ecgd/2000-04-03/S1088-4173-00-00050-3/S1088-4173-00-00050-3.pdf},
      review={\MR{1749249}},
}

\bib{MR1212727}{article}{
      author={Bru, Rafael},
      author={Johnson, Charles~R.},
       title={Nilpotence of products of nonnegative matrices},
        date={1993},
        ISSN={0035-7596},
     journal={Rocky Mountain J. Math.},
      volume={23},
      number={1},
       pages={5\ndash 16},
         url={https://doi-org.proxyiub.uits.iu.edu/10.1216/rmjm/1181072607},
      review={\MR{1212727}},
}

\bib{MR3289714}{incollection}{
      author={Buff, Xavier},
      author={Cui, Guizhen},
      author={Tan, Lei},
       title={Teichm\"{u}ller spaces and holomorphic dynamics},
        date={2014},
   booktitle={Handbook of {T}eichm\"{u}ller theory. {V}ol. {IV}},
      series={IRMA Lect. Math. Theor. Phys.},
      volume={19},
   publisher={Eur. Math. Soc., Z\"{u}rich},
       pages={717\ndash 756},
         url={https://doi-org.proxyiub.uits.iu.edu/10.4171/117-1/17},
      review={\MR{3289714}},
}

\bib{bekp}{incollection}{
      author={Buff, Xavier},
      author={Epstein, Adam},
      author={Koch, Sarah},
      author={Pilgrim, Kevin},
       title={On {T}hurston's pullback map},
        date={2009},
   booktitle={Complex dynamics},
   publisher={A K Peters},
     address={Wellesley, MA},
       pages={561\ndash 583},
      review={\MR{2508269 (2010g:37071)}},
}

\bib{cfpp:net}{article}{
      author={Cannon, J.~W.},
      author={Floyd, W.~J.},
      author={Parry, W.~R.},
      author={Pilgrim, K.~M.},
       title={Nearly {E}uclidean {T}hurston maps},
        date={2012},
        ISSN={1088-4173},
     journal={Conform. Geom. Dyn.},
      volume={16},
       pages={209\ndash 255},
         url={http://dx.doi.org/10.1090/S1088-4173-2012-00248-2},
      review={\MR{2958932}},
}

\bib{Cicone:2015aa}{article}{
      author={Cicone, Antonio},
       title={A note on the joint spectral radius},
        date={201502},
      eprint={1502.01506},
         url={https://arxiv.org/pdf/1502.01506.pdf},
}

\bib{cgnpp}{article}{
      author={Cordwell, Kristin},
      author={Gilbertson, Selina},
      author={Nuechterlein, Nicholas},
      author={Pilgrim, Kevin~M.},
      author={Pinella, Samantha},
       title={On the classification of critically fixed rational maps},
        date={2015},
     journal={Conformal geometry and dynamics},
      volume={19},
       pages={51\ndash 94},
}

\bib{DH1}{article}{
      author={Douady, A.},
      author={Hubbard, John},
       title={A proof of {T}hurston's topological characterization of rational
  functions},
        date={1993},
     journal={Acta. Math.},
      volume={171},
      number={2},
       pages={263\ndash 297},
}

\bib{MR1813237}{article}{
      author={Farb, Benson},
      author={Lubotzky, Alexander},
      author={Minsky, Yair},
       title={Rank-1 phenomena for mapping class groups},
        date={2001},
     journal={Duke Math. J.},
      volume={106},
       pages={581\ndash 597},
      review={\MR{2471596}},
}

\bib{fm}{book}{
      author={Farb, Benson},
      author={Margalit, Dan},
       title={A primer on mapping class groups},
   publisher={Princeton Univ. Press, Princeton},
        date={2012},
      review={\MR{2850125}},
}

\bib{netmapsite}{misc}{
      author={Floyd, William},
       title={The {NET} map website},
        note={\url{https://intranet.math.vt.edu/netmaps/}},
}

\bib{kmp:origami}{article}{
      author={Floyd, William},
      author={Kelsey, Gregory},
      author={Koch, Sarah},
      author={Lodge, Russell},
      author={Parry, Walter},
      author={Pilgrim, Kevin~M.},
      author={Saenz, Edgar},
       title={Origami, affine maps, and complex dynamics},
        date={2016},
     journal={Arnold Mathematical Journal},
      volume={3},
      number={3},
       pages={365\ndash 395},
      eprint={1612.06449},
         url={https://arxiv.org/abs/1612.06449},
}

\bib{floyd2022realizingpolynomialportraits}{misc}{
      author={Floyd, William},
      author={Kim, Daniel},
      author={Koch, Sarah},
      author={Parry, Walter},
      author={Saenz, Edgar},
       title={Realizing polynomial portraits},
        date={2022},
         url={https://arxiv.org/abs/2105.10055},
}

\bib{Floyd:2017kq}{article}{
      author={Floyd, William},
      author={Parry, Walter},
      author={Pilgrim, Kevin~M.},
       title={Modular groups, {H}urwitz classes and dynamic portraits of {NET}
  maps},
        date={2019},
     journal={Groups Geometry and Dynamics},
      volume={13},
       pages={47\ndash 88},
      eprint={1703.03983},
         url={https://arxiv.org/abs/1703.03983},
}

\bib{Floyd:2017fj}{article}{
      author={Floyd, William},
      author={Parry, Walter},
      author={Pilgrim, Kevin~M.},
       title={Presentations of {NET} maps},
        date={2019},
     journal={Fundamenta Mathematica},
      volume={244},
       pages={49\ndash 72},
      eprint={1701.00443},
         url={https://arxiv.org/abs/1701.00443},
}

\bib{MR4278340}{article}{
      author={Floyd, William},
      author={Parry, Walter},
      author={Pilgrim, Kevin~M.},
       title={Rationality is decidable for nearly {E}uclidean {T}hurston maps},
        date={2021},
        ISSN={0046-5755},
     journal={Geom. Dedicata},
      volume={213},
       pages={487\ndash 512},
         url={https://mathscinet.ams.org/mathscinet-getitem?mr=4278340},
      review={\MR{4278340}},
}

\bib{MR2471596}{article}{
      author={Hamenst\"adt, Ursula},
       title={Geometry of the mapping class groups. {I}. {B}oundary
  amenability},
        date={2009},
     journal={Invent. Math.},
      volume={175},
       pages={545\ndash 609},
      review={\MR{2471596}},
}

\bib{HP}{misc}{
      author={Hlushchanka, Mikhail},
      author={Prochorov, Nikolai},
       title={Critically fixed {T}hurston maps: classification, recognition,
  and twisting},
        date={(2022)},
        note={preprint, available at \url{https://arXiv.org/abs/2212.14759}},
}

\bib{MR3675959}{book}{
      author={Hubbard, John~Hamal},
       title={Teichm\"{u}ller theory and applications to geometry, topology,
  and dynamics},
   publisher={Matrix Editions, Ithaca, NY},
        date={2016},
        ISBN={978-1-943863-00-6},
  url={https://mathscinet-ams-org.proxyiub.uits.iu.edu/mathscinet-getitem?mr=3675959},
      review={\MR{3675959}},
}

\bib{kameyama:equivalence}{article}{
      author={Kameyama, Atsushi},
       title={The {T}hurston equivalence for postcritically finite branched
  coverings},
        date={2001},
        ISSN={0030-6126},
     journal={Osaka J. Math.},
      volume={38},
      number={3},
       pages={565\ndash 610},
      review={\MR{1 860 841}},
}

\bib{koch:endos}{article}{
      author={Koch, Sarah},
       title={Teichm{\"u}ller theory and critically finite endomorphisms},
        date={2013},
        ISSN={0001-8708},
     journal={Adv. Math.},
      volume={248},
       pages={573\ndash 617},
         url={http://dx.doi.org/10.1016/j.aim.2013.08.019},
      review={\MR{3107522}},
}

\bib{kmp:kps}{article}{
      author={Koch, Sarah},
      author={Pilgrim, Kevin~M.},
      author={Selinger, Nikita},
       title={Pullback invariants of {T}hurston maps},
        date={2016},
        ISSN={0002-9947},
     journal={Trans. Amer. Math. Soc.},
      volume={368},
      number={7},
       pages={4621\ndash 4655},
         url={http://dx.doi.org/10.1090/tran/6482},
      review={\MR{3456156}},
}

\bib{MR3054977}{book}{
      author={Lodge, Russell},
       title={Boundary values of the {T}hurston pullback map},
   publisher={ProQuest LLC, Ann Arbor, MI},
        date={2012},
        ISBN={978-1267-47157-4},
  url={http://gateway.proquest.com.proxyiub.uits.iu.edu/openurl?url_ver=Z39.88-2004&rft_val_fmt=info:ofi/fmt:kev:mtx:dissertation&res_dat=xri:pqm&rft_dat=xri:pqdiss:3517104},
        note={Thesis (Ph.D.)--Indiana University},
      review={\MR{3054977}},
}

\bib{nekrashevych:book:selfsimilar}{book}{
      author={Nekrashevych, Volodymyr},
       title={Self-similar groups},
      series={Mathematical Surveys and Monographs},
   publisher={American Mathematical Society},
     address={Providence, RI},
        date={2005},
      volume={117},
        ISBN={0-8218-3831-8},
      review={\MR{MR2162164}},
}

\bib{nekrashevych:combinatorics}{incollection}{
      author={Nekrashevych, Volodymyr},
       title={Combinatorics of polynomial iterations},
        date={2009},
   booktitle={Complex dynamics},
   publisher={A K Peters},
     address={Wellesley, MA},
       pages={169\ndash 214},
      review={\MR{2508257}},
}

\bib{MR3199801}{article}{
      author={Nekrashevych, Volodymyr},
       title={Combinatorial models of expanding dynamical systems},
        date={2014},
        ISSN={0143-3857},
     journal={Ergodic Theory Dynam. Systems},
      volume={34},
      number={3},
       pages={938\ndash 985},
         url={https://mathscinet.ams.org/mathscinet-getitem?mr=3199801},
      review={\MR{3199801}},
}

\bib{MR4475129}{book}{
      author={Nekrashevych, Volodymyr},
       title={Groups and topological dynamics},
      series={Graduate Studies in Mathematics},
   publisher={American Mathematical Society, Providence, RI},
        date={[2022] \copyright 2022},
      volume={223},
        ISBN={[9781470463809]; [9781470471194]},
         url={https://mathscinet.ams.org/mathscinet-getitem?mr=4475129},
      review={\MR{4475129}},
}

\bib{pf4site}{misc}{
      author={Parry, Walter},
       title={The {Pf4} computer program},
        note={\url{https://sourceforge.net/projects/pf4/}},
}

\bib{ramadas2025thurstonobstructionstropicalgeometry}{misc}{
      author={Ramadas, Rohini},
       title={Thurston obstructions and tropical geometry},
        date={2025},
         url={https://arxiv.org/abs/2402.14421},
        note={available at \url{https://arxiv.org/pdf/2402.14421}},
}

\bib{selinger:pullback}{article}{
      author={Selinger, Nikita},
       title={Thurston's pullback map on the augmented {T}eichm\"uller space
  and applications},
        date={2012},
        ISSN={0020-9910},
     journal={Invent. Math.},
      volume={189},
      number={1},
       pages={111\ndash 142},
         url={http://dx.doi.org/10.1007/s00222-011-0362-3},
      review={\MR{2929084}},
}

\end{biblist}
\end{bibdiv}

\end{document}